\documentclass[dvipsnames,onefignum,onetabnum]{siamart190516}
\usepackage[T1]{fontenc}
\usepackage{lmodern}
\usepackage[top=1.4in,bottom=1.275in,left=1.4in,right=1.4in]{geometry}
\usepackage[english]{babel}%
\usepackage{csquotes}%
\usepackage[activate={true,nocompatibility},tracking=true]{microtype}%
\usepackage{hyphenat}%
\usepackage{siunitx}
\usepackage{placeins}
\usepackage{cite}%
\usepackage[begintext=``, endtext='']{quoting}
\SetBlockEnvironment{quoting}
\SetBlockThreshold{1}
\usepackage{enumitem}
\usepackage{comment}

\usepackage{amsmath}
\usepackage{amsfonts}
\usepackage{amssymb}
\usepackage{amscd}
\usepackage{mathtools}
\usepackage{mathrsfs}
\usepackage{calligra}
\DeclareMathAlphabet{\mathcalligra}{T1}{calligra}{m}{n}
\DeclareFontShape{T1}{calligra}{m}{n}{<->s*[1.1]callig15}{}
\usepackage{bm}
\usepackage{autonum}%
\usepackage[nice]{nicefrac}
\usepackage{bbm}

\usepackage{tikz}
\usepackage{pgf}
\usepackage{pgfplots}
\usepackage{pgfplotstable}
\usepackage{shellesc}
\pgfplotsset{compat=newest}
\usetikzlibrary{spy}
\usetikzlibrary{shapes}
\usetikzlibrary{backgrounds}
\usetikzlibrary{shadows}
\usetikzlibrary{matrix}
\usetikzlibrary{external}
\usepgfplotslibrary{fillbetween}
\usepackage{graphicx}
\usepackage{import}
\usepackage{cleveref}
\crefname{equation}{equation}{equations}
\crefname{theorem}{theorem}{theorems}
\crefname{chapter}{chapter}{chapters}
\crefname{figure}{figure}{figures}
\crefname{conject}{conjecture}{conjectures}
\crefname{lemma}{lemma}{lemma}
\crefname{corollary}{corollary}{corollaries}
\crefname{section}{section}{sections}
\crefname{definition}{definition}{definitions}
\crefname{condition}{condition}{conditions}
\crefname{test}{test}{tests}
\crefname{table}{table}{tables}
\crefname{remark}{remark}{remarks}
\crefname{problem}{problem}{problems}
\crefname{proposition}{proposition}{propositions}
\crefname{algocf}{algorithm}{algorithms}

\usepackage{csvsimple}
\usepackage{booktabs}
\usepackage{environ}
\usepackage{relsize}
\usepackage[font=small,figurewithin=section]{caption}%
\usepackage[subrefformat=parens,labelformat=simple]{subcaption}%

\usepackage{afterpage}
\usepackage[ruled,vlined,algo2e]{algorithm2e}
\newtheorem{condition}{Condition}%
\newtheorem{test}{Test}%

\let\inf\relax \DeclareMathOperator*\inf{\vphantom{p}inf}
\let\min\relax \DeclareMathOperator*\min{\vphantom{p}min}
\let\max\relax \DeclareMathOperator*\max{\vphantom{p}max}
\let\subset\relax \DeclareMathOperator{\subset}{\subseteq}
 
\let\tilde\widetilde

\DeclareMathOperator*{\argmin}{arg\,min}

\newcommand{\E}{\mathbb{E}} %
\newcommand{\R}{\mathbb{R}} %
\newcommand{\dd}{\,\mathrm{d}} %

\newcommand{\sfN}{\mathsf{N}}

\newcommand{\mcO}{\mathcal{O}}

\newcommand{\mcR}{\mathcal{R}}

\newcommand{\scB}{\mathscr{B}}

\newcommand{\bbE}{\mathbb{E}}

\newcommand{\bbP}{\mathbb{P}}

\newcommand{\bfe}{\mathbf{e}}
\newcommand{\bff}{\mathbf{f}}

\makeatletter
\newsavebox{\measure@tikzpicture}
\NewEnviron{scaletikzpicturetowidth}[1]{%
	\tikzifexternalizingnext{%
		\def\tikz@width{#1}%
		\def\tikzscale{1}\begin{lrbox}{\measure@tikzpicture}%
			\tikzset{external/export next=false,external/optimize=false}%
			\BODY
		\end{lrbox}%
		\pgfmathparse{#1/\wd\measure@tikzpicture}%
		\edef\tikzscale{\pgfmathresult}%
		\BODY
	}{%
		\BODY
	}%
}
\makeatother

\definecolor{color0}{rgb}{0.7843, 0.7843, 0.7843}
\definecolor{color1}{rgb}{0, 0.4470, 0.7410}
\definecolor{color2}{rgb}{0.8500, 0.3250, 0.0980}
\definecolor{color3}{rgb}{0.9290, 0.6940, 0.1250}
\definecolor{color4}{rgb}{0.7060, 0.3840, 0.7650}
\definecolor{color5}{rgb}{0.4660, 0.6740, 0.1880}
\definecolor{color6}{rgb}{0.3010, 0.7450, 0.9330}
\definecolor{color7}{rgb}{0.6350, 0.0780, 0.1840}
\definecolor{color8}{rgb}{0.0, 0.4078, 0.3412}

\pgfplotsset{
  log x ticks with fixed point/.style={
      xticklabel={
        \pgfkeys{/pgf/fpu=true}
        \pgfmathparse{exp(\tick)}%
        \pgfmathprintnumber[fixed relative, precision=3]{\pgfmathresult}
        \pgfkeys{/pgf/fpu=false}
      }
  },
  log y ticks with fixed point/.style={
      yticklabel={
        \pgfkeys{/pgf/fpu=true}
        \pgfmathparse{exp(\tick)}%
        \pgfmathprintnumber[fixed relative, precision=3]{\pgfmathresult}
        \pgfkeys{/pgf/fpu=false}
      }
  }
}

\newcommand{\CVaR}{\ensuremath{\mathrm{CVaR}}}

\newcommand{\var}{\ensuremath{x}}
\newcommand{\vardim}{\ensuremath{n}}
\newcommand{\varopt}{\ensuremath{\var^*}}
\newcommand{\its}{\ensuremath{k}}
\newcommand{\varit}{\ensuremath{\var_\its}}
\newcommand{\varitnext}{\ensuremath{\var_{\its+1}}}
\newcommand{\rv}{\ensuremath{\xi}}
\newcommand{\obj}{\ensuremath{f}}
\newcommand{\Obj}{\ensuremath{F}}

\newcommand{\gradC}{R}

\newcommand{\subsampling}{\ensuremath{{S_\its}}}
\newcommand{\gradCsub}{\ensuremath{\gradC_{\subsampling}}}

\newcommand{\proj}{P}
\pgfplotsset{
    legend image with text/.style={
        legend image code/.code={%
            \node[anchor=center] at (0.3cm,0cm) {#1};
        }
    },
}

\pgfplotsset{select coords between index/.style 2 args={
    x filter/.code={
        \ifnum\coordindex<#1\def\pgfmathresult{}\fi
        \ifnum\coordindex>#2\def\pgfmathresult{}\fi
    }
}}

\ifpdf
  \DeclareGraphicsExtensions{.eps,.pdf,.png,.jpg}
\else
  \DeclareGraphicsExtensions{.eps}
\fi

\newsiamremark{remark}{Remark}
\newsiamremark{hypothesis}{Hypothesis}
\crefname{hypothesis}{Hypothesis}{Hypotheses}
\newsiamthm{claim}{Claim}

\headers{Adaptive Sampling for Stochastic Optimization with Constraints}{F. Beiser, B. Keith, S. Urbainczyk, and B. Wohlmuth}

\title{Adaptive Sampling Strategies for Risk-Averse Stochastic Optimization with Constraints\thanks{Submitted to the editors \today.
\funding{This project received funding from the European Union's Horizon 2020 research and innovation programme under grant agreement No. 800898 as well as partial support from the German Research Foundation by grant WO671/11-1. In addition, much of the manuscript was written while the second author was in residence at the Institute for Computational and Experimental Research in Mathematics (ICERM) in Providence, RI, during the Advances in Computational Relativity program, supported by the National Science Foundation under Grant No. DMS-1439786. The first author gratefully acknowledges support from the International Research Training Group IGDK, funded by the German Science Foundation (DFG) and the Austrian ScienceFund (FWF).}}}
\author{
	Florian~Beiser\thanks{\protect
		Mathematics and Cybernetics, SINTEF Digital, Forskningsveien 1, 0373 Oslo, NO (\email{florian.beiser@sintef.no})
	}%
	\and Brendan~Keith\thanks{\protect
		Division of Applied Mathematics,
    Brown University,
    Providence, RI 02912 USA
    (\email{brendan\_keith@brown.edu}).
	}%
	\and Simon~Urbainczyk\thanks{\protect
		Maxwell Institute for Mathematical Sciences and Department of Actuarial Mathematics and Statistics, Heriot-Watt University, Edinburgh EH14 4AS, UK (\email{su2004@hw.ac.uk})
	}%
	\and Barbara~Wohlmuth\thanks{\protect
		Chair for Numerical Mathematics, Department of Mathematics, Technical University of Munich, Boltzmannstra{\ss}e 3, 80333 Munich, DE (\email{wohlmuth@ma.tum.de})
	}%
}

\usepackage{amsopn}

\newcommand{\revised}[1]{{#1}}
\newcommand{\revisedBW}[1]{{#1}}

\graphicspath{{./Figures/}}

\allowdisplaybreaks

\begin{document}

\maketitle

\begin{abstract}
\revised{We introduce adaptive sampling methods for stochastic programs with deterministic constraints.
First, we propose and analyze a variant of the stochastic projected gradient method where the sample size used to approximate the reduced gradient is determined on-the-fly and updated adaptively.
This method is applicable to a broad class of expectation-based risk measures and leads to a significant reduction in the individual gradient evaluations used to estimate the objective function gradient.
Numerical experiments with expected risk minimization and conditional value-at-risk minimization support this conclusion and demonstrate practical performance and efficacy for both risk-neutral and risk-averse problems.
Second, we propose an SQP-type method based on similar adaptive sampling principles.
The benefits of this method are demonstrated in a \revisedBW{simplified} engineering design application featuring risk-averse shape optimization of a steel shell structure subject to uncertain loading conditions and model uncertainty.}
\end{abstract}



\section{Introduction} %
\label{sec:introduction}

In this article, we consider the following general class of stochastic programs:
\begin{equation}
\min_{\var\in C}~ \Big\{ F(\var) = \mathcal{R}[\obj(\var;\rv)] \Big\}
.
\label{eq:P_general}
\end{equation}
Here, $\obj:\R^n\to\R$ is a smooth function, $\rv$ is a random variable on a probability space denoted $(\Xi,\scB,\bbP)$, $C\subset \R^n$ is a closed subdomain, and $\mcR:L^1(\Xi,\scB,\bbP)\to \R$ is a \emph{coherent risk measure} \cite{artzner1999coherent}.
A canonical example of the objective function $F(\var)$ is the expected value of $f$ at $\var$; namely,
\begin{equation}
	\bbE[\obj(\var;\rv)]
	=
	\int_\Xi \obj(\var;\rv) \dd \bbP(\rv)
	.
\label{eq:ExpectedValue}
\end{equation}
In this situation, $\mcR = \bbE$ and we say that $F(x) = \bbE[\obj(\var;\rv)]$ defines the expected risk at $x$.

With the expected value risk measure, $\mcR = \bbE$, the stochastic program~\cref{eq:P_general} only seeks out the point $x = x^\ast$ which minimizes $f(x)$ on average.
More general risk measures are often used when it is desirable to optimize for low probability events.
The other risk measure considered in this work is the conditional value-at-risk ($\CVaR$) \cite{Rockafellar2000,rockafellar2002conditional}.
The $\CVaR$ at confidence level $\beta\in (0,1)$, denoted $\CVaR_\beta$, is a well-established decision-making tool in finance \cite{krokhmal2002portfolio,Shapiro2009} and is becoming increasingly prominent in engineering \cite{TyrrellRockafellar2015,kouri2016risk,yang2017algorithms,Kouri2018a,chaudhuri2020risk,chaudhuri2020multifidelity}.
In this work, stochastic programs featuring the risk measure $\mcR = \bbE$ are referred to as \emph{risk-neutral}, meanwhile, those involving the risk measure $\mcR = \CVaR_\beta$, are referred to as \emph{risk-averse}.

\revised{In many practical problems, the integral in~\cref{eq:ExpectedValue} cannot be computed directly because the space $\Xi$ is high-dimensional \cite{Kouri2018,bottou2018optimization}.}
One common approach to approximate the integral is to draw a set of i.i.d. samples $S = \{\xi_i\}$, $i=1,\ldots,N$, of the random variable $\xi$ and substitute $F(x)$ with the following empirical estimate of the expected risk:
\begin{equation}
	F_S(\var)
	=
	\frac{1}{N}
	\sum_{i=1}^N
	\obj(\var;\rv_i)
	.
\label{eq:ExpectedValue_samples}
\end{equation}
When the sample set $S$ is fixed, stochastic optimization methods which employ this type of approximation of the objective function are commonly referred to as sample average approximation methods \cite{Shapiro2009,Royset2013,Kouri2018}.
\revised{In this paper, we design algorithms that \emph{adaptively} determine the \revisedBW{size of the} set $S$ in each iteration.}

\revised{Since in high-dimensional settings the computational cost of optimizing \cref{eq:P_general} typically hinges on the total number of gradient evaluations of $f$, we aim at keeping this number as small as possible. To this end,}
we propose a sampling strategy which adaptively balances the algorithm's sampling error and optimization error throughout the entire optimization process.
The strategy works by updating the size of the sample set $S = S_k$ together with the point $x = x_{k}$.
More precisely, at each iteration $k$, we generate a sample set $S_k$ based on gradient evaluations at $x_k$ to compute an updated point $x_{k+1}$.
This leads to robust and practical methods that can treat many
stochastic optimization problems efficiently.

\subsection{Literature review and motivation}
\label{subsec:literature}

There are many articles on stochastic optimization methods with dynamic sample sizes \cite{homem2003variable,friedlander2012hybrid,byrd2012sample,Royset2013,kouri2013trust,cartis2018global,pasupathy2018sampling,roosta2019sub,de2017automated,Bollapragada2018,Bollapragada2018a,bottou2018optimization,Bollapragada2019,paquette2020stochastic}.
Nevertheless, very few of these works consider \emph{constrained} optimization problems or risk-averse settings in detail \cite{Royset2013}; the majority of the present literature focuses on \emph{unconstrained} stochastic programs, such as those commonly found in machine learning.
One notable exception is the recent contributions by Xie et al. \cite{xie2020constrained,xie2021methods}, which appeared online shortly after an earlier version of this work \cite{beiser2020adaptive_v1} and complements our contribution by, among other novelties, introducing alternative adaptive sampling strategies with separate convergence proofs as well as analyzing composite optimization problems.
A thorough comparison of this work and \cite{xie2020constrained,xie2021methods} is given at multiple points later in the text; see~\Cref{rem:ComparisontoXie1,rem:ComparisontoXie2,rem:ComparisontoXie3}.

The difficulty in generalizing previous work on adaptive sampling to constrained optimization problems lies in developing new critera to quantify and balance the statistical and optimization errors, while accounting for the influence of the constraint set.
We refer the interested reader to \cite[Section~1]{xie2020constrained} for an overview of the pitfalls of applying well-established adaptive sampling strategies designed for unconstrained problems to the constrained setting.
For constrained stochastic programs, most contemporary methods rely on \textit{a priori} error analysis that results in a prescribed growth in the sample size \cite{friedlander2012hybrid,Royset2013,Bollapragada2019}.
\revised{In this work, as was}
also done in \cite{xie2020constrained,xie2021methods}, we choose to estimate the correct sample size \textit{a posteriori} and update it adaptively.

Our work has a great deal in common with the adaptive sampling approaches taken in \cite{byrd2012sample,Bollapragada2018,Bollapragada2018a}. %
In order to highlight the primary similarities, we note that in our approach to (stochastic) projected gradient descent (\Cref{sub:idealized_algorithm,sub:a_less_than_ideal_algorithm} and \Cref{sub:a_practical_algorithm}), we arrive at a condition similar to the ``norm test'' introduced for unconstrained optimization in \cite{carter1991on}, and later used in \cite{byrd2012sample}.\footnote{The words ``norm test'' are not actually used in either \cite{carter1991on} or \cite{byrd2012sample}, but recent works by the authors of \cite{byrd2012sample} have promoted this terminology; see, e.g., \cite{Bollapragada2018a,xie2020constrained}.}
A similar test also appears in our sequential quadratic programming (SQP) algorithm (\revised{cf.} \Cref{sec:nonconvex-applications}).

Another approach to deal with stochastic programs with deterministic constraints appeared online a few months after the initial version of this article \cite{na21,na21_b}. While \cite{xie2020constrained} present an adaptive sampling algorithm that shows similar properties to our methods, Na et al. \cite{na21} use an indepedent approach to develop a novel stochastic line search procedure and associated stochastic SQP algorithm that can also be extended to work for inequality-constrained stochastic programs based on active-set strategies \cite{na21_b}.

The present work arose from a need to develop efficient stochastic programming methods for large-scale decision-making problems; especially in engineering design, where each individual sample computation is extremely costly \cite{kouri2013trust,Shi2018,Geiersbach2020,Ion2018,zou2019adaptive}.
In these high-cost scenarios, one wishes to evaluate as few samples as possible.
It is well-established that the expected risk~\cref{eq:ExpectedValue} is often unsuitable to predict immediate and long-term performance, manufacturing and maintenance costs, system response, levels of damage, and numerous other quantities of interest \cite{Rockafellar2010,kouri2016risk,Kouri2018,Kouri2018a}.
Therefore, today's industrial problems are made even more challenging because they typically require a risk-averse formulation \cite{kodakkal2022riskaverse,TyrrellRockafellar2015}.

\subsection{Layout}
Apart from the expected value operator $\bbE$, the conditional value-at-risk is the only risk measure we consider in detail.
It is well-known that this risk measure can be reformulated as a separate optimization problem involving $\bbE$; cf. \cite{Rockafellar2000,rockafellar2002conditional} and \Cref{sec:risk_averse_problems}. 
This observation informs the layout of the paper by allowing us to first focus on the case $\mathcal{R}=\E$ and then deal with the treatment of risk-averse problems in the later sections.
A large family of other important risk measures, including the entropic risk and the conditional entropic risk \cite{Kouri2018a}, have a similar reformulation involving $\bbE$ \cite{rockafellar2013fundamental,Kouri2018a}, which leads us to conclude that there is little loss of generality in treating~\cref{eq:P_general} in this incremental and case-specific way.
Likewise, in order to develop our sample size conditions and then analyze the corresponding algorithms, we begin with a \emph{convex} constraint set $C\subset \mathbb{R}^n$.
\revised{Later on, we give an outlook on how the algorithm can be generalized to treat non-convex equality constraints.}

In~\Cref{sec:adaptive_sampling}, we use the expected risk problem to introduce basic adaptive sampling principles for stochastic projected gradient descent (SPGD) and theoretical conditions which imply convergence.
We then use these conditions in \Cref{sub:a_practical_algorithm} to propose a simple SPGD algorithm that solves the expected risk problem with convex constraints.
In \Cref{sec:risk_averse_problems}, we present two ways in which these algorithms may be extended to handle risk-averse problems.
Here, we focus on the $\CVaR$ risk measure and state consequences for other important risk measures only in passing.
\Cref{sec:applications} is dedicated to in-depth numerical studies which test the efficacy of our adaptive sampling method in its various forms.
\revised{Afterwards, in \Cref{sec:nonconvex-applications}, we generalize the algorithm using SQP principles and arrive at a new method appropriate for an important class of problems with non-convex constraints. This is complemented by a complex numerical example from engineering design.}
The paper then closes with a short summary of results.
Finally, additional numerical experiments are documented in \Cref{appendix:numerical_examples}.
\section{Adaptive sampling with convex constraints} %
\label{sec:adaptive_sampling}

In this and the following section, we only consider $\mcR = \bbE$.
This setting allows~\cref{eq:P_general} to be rewritten as
\begin{equation}
\min_{\var\in C}~ \Big\{ F(\var) = \bbE[\obj(\var;\rv)] \Big\}
.
\label{eq:P-C}
\end{equation}
For the time being, we also assume that $C\subset\R^n$ is convex.
To treat this problem, we propose a projected gradient descent algorithm and sufficient conditions on the sample sets $S_k$, which guarantee that it is a descent method in expectation.

\subsection{Preliminaries and notation} %
\label{sub:preliminaries_and_notation}

Let $\nabla F(\var)$ denote the gradient of $F$ at $\var$ and let $\langle \cdot, \cdot\rangle$ denote the $\ell^2$ inner product on vectors in $\R^n$.
It is well-known (see, e.g., \cite{Nesterov2018}) that if $F$ is both convex and continuously differentiable, then $x^\ast$ is a solution of~\cref{eq:P-C} if and only if
\begin{equation}
\label{eq:ConvexOptimalityCondition}
	\langle \nabla F(\varopt), \var - \varopt\rangle
	\geq 0
\end{equation}
for all $\var \in C$.

When $C = \R^n$, one may use the stochastic gradient descent algorithm, $\varitnext = \varit - \alpha \nabla \Obj_\subsampling(\varit)$, to uncover locally optimal solutions of~\cref{eq:P-C}; cf. \cite{byrd2012sample,Bollapragada2018a}.
Here, $\alpha > 0$ is a step-length parameter and $\nabla \Obj_\subsampling(\var_k)$ denotes the gradient of the sample average defined in~\cref{eq:ExpectedValue_samples} with an iteration-dependent sample set $S = S_k$.
When the convex set $C \neq \R^n$, the analogue of this approach is the stochastic projected gradient descent (SPGD) algorithm; $y_{k+1} = \varit - \alpha \nabla \Obj_\subsampling(\varit)$, $\varitnext = \argmin_{x\in C} \|y_{k+1} - x\|^2$, where $\|\cdot\|$ denotes the Euclidean norm.
Equivalently \cite{Nesterov2018}, we write
\begin{equation}
\label{eq:GeneralIteration}
	\varitnext
	=
	\argmin_{x\in C}
	\Big\{
	\Obj_\subsampling(\varit) + \langle \nabla \Obj_\subsampling(\varit),x-\varit \rangle + \frac{1}{2\alpha}\lVert x-\varit \rVert^2
	\Big\}
	.
\end{equation}

We will write $\mathbb{E}_k[\cdot]$ to denote the expected value operator~\cref{eq:ExpectedValue}, given $x_k$.
With this notation, quantities such as $\mathbb{E}_k[F(x_{k+1})]$ are well-defined because $x_{k+1}$ depends only on the random variable $\xi$ through~\cref{eq:ExpectedValue_samples,eq:GeneralIteration}.
We will also assume that $S_k$, for each $k$, is a set of i.i.d. samples, independent of each previous set $S_0,S_1,\ldots,S_{k-1}$.
With this assumption, $\nabla \Obj_\subsampling(\varit)$ forms a unbiased estimator for the gradient at $x_k$, namely,
\begin{equation}
	\mathbb{E}_k[\nabla F_{S_k}(x_{k})]
	=
	\nabla F(x_k)
	.
\label{eq:UnbiasedGradient}
\end{equation}
When we wish to analyze the \emph{total expectation} of an iteration-dependent quantity, say $\mathbb{E}[F(x_k)]$, we note that it is completely determined by the joint distribution of samples in $S_0,S_1,\ldots,S_{k-1}$.
For this reason, we have the identity \cite{bottou2018optimization}
\begin{equation}
	\mathbb{E}[F(x_k)]
	=
	\mathbb{E}_0\mathbb{E}_1 \cdots \mathbb{E}_{k-1}[F(x_k)]
	.
\label{eq:TotalExpectationIdentity}
\end{equation}

In the sequel, it will also be convenient to assign symbols to certain terms in the equations above.
First, we define the orthogonal projection onto $C$,
\begin{equation}
\label{eq:ProjectionOperator}
 	\proj(y)
 	=
 	\argmin_{x\in C}
 	\|y - x\|^2
 	.
\end{equation}
Note that because $C$ is convex, $\proj(y)$ is unique and non-expansive \cite[Corollary~2.2.3]{Nesterov2018}, namely,
\begin{equation}
\label{eq:non-expansive}
	\|\proj(x) - \proj(y)\|^2
	\leq
	\langle x-y, \proj(x) - \proj(y)\rangle
	\leq
	\|x - y\|^2
	,
\end{equation}
and generally it is \emph{non-linear}.
Next, we denote the projected gradient mapping, $Q:\R^\vardim\rightarrow C$, as $Q(\var) = \proj(\var-\alpha\nabla \Obj(\var))$.
Equivalently, one may write
\begin{equation}
	Q(\var)
	=
	\argmin_{y\in C}
	\Big\{
	\Obj(\var) + \langle \nabla \Obj(\var),y-\var \rangle + \frac{1}{2\alpha}\lVert y-\var \rVert^2
	\Big\}
	.
\label{eq:GradientMap}
\end{equation}
The subsampled gradient map is then defined analogously to~\cref{eq:GradientMap}; namely,
\begin{equation}
	Q_{S_k}(\var)
	=
	\argmin_{y\in C}
	\Big\{
	\Obj_\subsampling(\var) + \langle \nabla \Obj_\subsampling(\var),y-\var \rangle + \frac{1}{2\alpha}\lVert y-\var \rVert^2
	\Big\}
	.
\label{eq:GradientMap_sampled}
\end{equation}
With this notation in hand, one may note that $x_{k+1} = Q_{S_k}(x_k)$ by~\cref{eq:GeneralIteration}.

The reduced gradient, defined by
\begin{equation}
\label{eq:ReducedGradientDefinition}
	\gradC(\var)=\alpha^{-1}(\var - Q(\var)),
\end{equation}
is another important operator we will make judicious use of.
The subsampled reduced gradient is likewise defined $$\gradCsub(\var) = \alpha^{-1}(\var - Q_\subsampling(\var)).$$
Clearly, $x_{k+1} = x_k - \alpha R_{S_k}(x_k)$.
Moreover, when $C=\R^n$, one may note that $R(x_k) = \nabla F(x_k)$ and $R_{S_k}(x_k) = \nabla F_{S_k}(x_k)$.

We may now formulate the first-order optimality condition for~\eqref{eq:P-C} as follows \cite{Nesterov2018}:
\begin{equation}
	\text{If $\varopt$ is a minimizer of~\eqref{eq:P-C}, then $Q(\varopt) = \varopt$ and $\gradC(\varopt) = 0$}.
\end{equation}
We may also state two lemmas based on \cite{Nesterov2018}, which will be useful later on.
For reference, we say that $F$ is $L$-smooth if
\begin{equation}
	\|\nabla F(x) - \nabla F(y) \| \leq L\|x-y\|
	,
\label{eq:LSmooth}
\end{equation}
for all $x,y \in C$,
and we say that $F$ is $\mu$-strongly convex if
\begin{equation}
	F(y) \geq F(x) + \langle \nabla F(x) , y - x \rangle + \frac{\mu}{2}\| y - x \|^2
	.
\label{eq:MuStronglyConvex}
\end{equation}
The proof of~\Cref{lem:ConstrainedConvexOptResults} can be found in \cite[Corollary~2.3.2]{Nesterov2018}.
For the reader's convenience, we include the proof of~\Cref{lem:GradientBounds}.
\begin{lemma}
\label{lem:ConstrainedConvexOptResults}
Assume that $F$ is $L$-smooth, let $0 < \alpha \leq 1/L$ and let $C$ be convex.
If $F$ is convex, then the following inequality holds for all $x \in C$:
\begin{equation}
	F(Q(x)) - F(\var) \leq - \frac{\alpha}{2}\|\gradC(\var)\|^2
	.
\label{eq:LSmoothConvexityIdentity}
\end{equation}
If, moreover, $F$ is $\mu$-strongly convex, then it also holds that
\begin{equation}
	\frac{\mu}{2} \lVert \var - \varopt \rVert^2 + \frac{\alpha}{2} \lVert \gradC(\var) \rVert^2
	\leq
	\langle \gradC(\var), \var - \varopt \rangle
	.
\label{eq:StongConvexityIdentity}
\end{equation}
\end{lemma}

\begin{lemma}
\label{lem:GradientBounds}
	Let $F$ be $L$-smooth and let $C$ be closed and convex.
	For all $y\in C$ and $z\in \R^n$, it holds that
	\begin{equation}
	\label{eq:GradientBounds}
		\langle \gradC(z), y - Q(z)\rangle
		\leq
		\langle \nabla\Obj(z), y - Q(z)\rangle
		.
	\end{equation}
\end{lemma}

\begin{proof}
	Fix $z\in \R^n$ and consider $\phi(y) = \Obj(z) + \langle \nabla \Obj(z),y-z\rangle  + \frac{1}{2\alpha}\| y-z \|^2$.
	Note that $\phi$ is both convex and continuously differentiable and that $\nabla \phi(y) = \nabla \Obj(z) + \frac{1}{\alpha}(y-z)$.
	Therefore, following from the optimality condition~\cref{eq:ConvexOptimalityCondition} and the definition of $Q(\cdot)$ in~\cref{eq:GradientMap}, it holds that
	\begin{equation}
		\langle \nabla\phi(Q(z)), y - Q(z) \rangle
		=
		\langle \nabla\Obj(z) - \gradC(z), y - Q(z) \rangle
		\geq
		0
	\end{equation}
	for all $y\in C$.
\end{proof}

\subsection{Descent conditions} %
\label{sub:idealized_algorithm}

The goal of our adaptive sampling scheme is to balance sampling and optimization error.
One way to strike this balance is through state-dependent conditions which ensure that $\bbE_k\big[F(x_{k+1})\big] \leq F(x_k)$.
In~\Cref{thm:Descent}, we show that it is sufficient that each sample set $S_k$ satisfies only two idealized conditions.
However, we note that these conditions require foreknowledge of the exact gradient at each iterate, $\varit$.
(This detail is dealt with in \Cref{sub:a_practical_algorithm}.)
The descent conditions are:
\smallskip
\begin{condition}
\label{con:Norm}
	Control of the norm of the reduced gradient:
	\begin{equation}
	\label{eq:ExactReducedGradientControl}
		\bbE_k\left[\lVert \gradCsub(\varit)\rVert^2\right]
		\leq
		(1+\nu^2)
		\lVert \gradC(\varit) \rVert^2
		,
	\end{equation}
	for some fixed $\nu > 0$.
\end{condition}
\smallskip 
\begin{condition}
\label{con:Bias}
	Control of the bias in the projected gradient mapping:
	\begin{equation}
	\label{eq:ExactBiasControl}
		\|\bbE_k[Q_{S_k}(\varit) - Q(\varit)]\|
		\leq
		\frac{\gamma^2}{2}
		\lVert Q(\varit) - \varit \rVert^2
		,
	\end{equation}
	for some fixed $\gamma > 0$.
\end{condition}
\smallskip 

\begin{theorem}
\label{thm:Descent}
	Assume that $F$ is $L$-smooth~\cref{eq:LSmooth} and that the sequence of iterates $\{\varit\}_{k=0}^\infty$ is contained in an open set over which $\|\nabla F(x)\|$ is bounded above by some constant $M>0$.
	If $S_k$ satisfies \Cref{con:Norm,con:Bias} and $\alpha \leq \frac{2}{M\gamma^2 + L(1 + \nu^2)}$, then
	\begin{equation}
		\bbE_k\big[F(x_{k+1})\big]
		\leq
		F(x_k)
		.
	\end{equation}
\end{theorem}
\begin{proof}
		By standard arguments following from the $L$-smoothness of $F$ \cite{bertsekas2015convex}, we have that
	\begin{align}
		F(x_{k+1})
		&\leq
		F(x_k)
		+
		\langle \nabla F(x_k), x_{k+1} - x_k \rangle
		+
		\frac{L}{2}\|x_{k+1} - x_k\|^2
		\\
		&=
		F(x_k)
		+
		\langle \nabla F(x_k), Q(x_k) - x_k \rangle
		+
		\langle \nabla F(x_k), E_{S_k}(x_k) \rangle
		+
		\frac{L}{2}\|x_{k+1} - x_k\|^2
		,
	\label{eq:DescentTheorem_eq1}
	\end{align}
	where $E_{S_k}(x_k) = x_{k+1} - Q(x_k)$.
	Now, substituting $y = z = x_k$ into~\cref{eq:GradientBounds}, we arrive at the identity
	\begin{equation}
		\langle \nabla F(x_k), Q(x_k) - x_k \rangle
		\leq
		-\langle \gradC(x_k), x_k - Q(x_k)\rangle
		=
		-
		\frac{1}{\alpha}
		\|x_k - Q(x_k)\|^2
		.
	\label{eq:DescentTheorem_eq2}
	\end{equation}
	Next, \Cref{con:Bias} implies that
	\begin{equation}
		\bbE_k\big[\langle\nabla\Obj(\varit), E_{S_k}(x_k) \rangle\big]
		\leq
		\|\nabla\Obj(\varit)\| \|\bbE_k[E_{S_k}(x_k)]\|
		\leq
		\frac{M}{2}\gamma^2
		\lVert Q(x_k) - x_k \rVert^2
		.
	\label{eq:DescentTheorem_eq3}
	\end{equation}
	Moreover, \Cref{con:Norm} implies that
	\begin{equation}
		\bbE_k\big[\|x_{k+1} - x_k\|^2\big]
		=
		\alpha^2 \bbE_k\big[\|R_{S_k}(x_k)\|^2\big]
		\leq
		\alpha^2(1+\nu^2) \|R(x_k)\|^2
		=
		(1+\nu^2) \|Q(x_k) - x_k\|^2
		.
	\label{eq:DescentTheorem_eq4}
	\end{equation}
	Combining~\cref{eq:DescentTheorem_eq1,eq:DescentTheorem_eq2,eq:DescentTheorem_eq3,eq:DescentTheorem_eq4}, we arrive at
	\begin{align}
	\label{eq:DescentTheorem_eq5}
		\bbE_k[F(x_{k+1})] - F(x_k)
		&\leq
		-c
		\|Q(x_k) - x_k\|^2
		,
	\end{align}
	where $c = \frac{1}{\alpha} - \frac{L}{2}(1+\nu^2) - \frac{M}{2}\gamma^2$.
	Note that if $\alpha \leq \frac{2}{M\gamma^2 + L(1 + \nu^2)}$, then the right-hand side of~\cref{eq:DescentTheorem_eq5} is non-positive.
\end{proof}

Although \Cref{con:Norm,con:Bias} are simple to write out, it is unfortunately difficult to design practical algorithms \revisedBW{that guarantee them strictly}.
\revised{
This difficulty is due in part to the presence of the projection operator $P\colon \R^n \to C$ \revisedBW{inside} the expected values on the left-hand sides of~\cref{eq:ExactReducedGradientControl,eq:ExactBiasControl}.
Therefore, checking these conditions would require numerous applications of $P$, which may be prohibitively expensive.
To avoid this difficulty, we turn to an alternative condition in the next subsection.
}

\revised{
\begin{remark}
	It is interesting to relate \Cref{con:Norm,con:Bias} to the analysis of stochastic gradient descent with adaptive sampling for unconstrained problems.
	In doing so, \Cref{con:Norm} can be viewed as a generalization of \cite[Equation~3.5]{Bollapragada2018a}, which is one of the key inequalities in that work.
	Moreover, we note that \Cref{con:Bias} is trivially satisfied for all $\gamma\geq 0$, whenever $C$ is \revisedBW{an affine subspace} of $\R^n$.
\end{remark}
}

\begin{remark}
	The parameters in \Cref{con:Norm,con:Bias} are defined so that if $f\colon \mathbb{R}^n\to\mathbb{R}$ is a deterministic function, then both conditions hold for all non-negative parameter values $\nu,\gamma \geq 0$.
	\revised{
	Moreover, by setting $\nu=\gamma=0$, we can recover~\Cref{lem:ConstrainedConvexOptResults}.
	To observe this fact, we must inspect the proof of \Cref{thm:Descent} and, in particular, inequality~\cref{eq:DescentTheorem_eq5}.
	Here, we see that if $\alpha \leq \frac{1}{M\gamma^2 + L(1 + \nu^2)}$, then $c \geq \frac{1}{2\alpha}$.
	With this stronger condition, we may write
	\begin{equation}
	\label{eq:DescentRemark}
		\bbE_k[F(x_{k+1})] - F(x_k)
		\leq
		-\frac{\alpha}{2}
		\|R(x_k)\|^2
		,
	\end{equation}
	which is analogous to~\cref{eq:LSmoothConvexityIdentity} and equivalent to \cref{eq:LSmoothConvexityIdentity} $f\colon \mathbb{R}^n\to\mathbb{R}$ is a deterministic function and $\nu=\gamma=0$.
	Similar bounds on the step size $\alpha$ will appear again.
	In anticipation of these expressions, we adopt the notation
	\begin{equation+}
		\tilde{L} = M \gamma^2 + L (1+\nu^2)
		.
	\end{equation+}
	}
\end{remark}

\begin{remark}
	The reader may notice that the normed quantity on the right-hand side of~\Cref{con:Bias} may be rewritten by definition~\cref{eq:ReducedGradientDefinition} as
	\begin{equation}
	\label{eq:SimpleIdentity}
		\|Q(x_k) - x_k\|^2 = \alpha^2 \|R(x_k)\|^2
		.
	\end{equation}
	This is a useful identity that we will rely on in the sequel.
\end{remark}

\subsection{Alternative condition} %
\label{sub:a_less_than_ideal_algorithm}

Let us focus on the bias condition given by~\cref{eq:ExactBiasControl}.
It may appear odd that its left-hand side involves a norm and its right-hand side involves a norm squared.
However, the bias term on the left-hand side is not absolutely homogeneous with respect to $\nabla F(x)$.
This is easily seen in the specific case where the boundary of the constraint set $C$ is smooth and, therefore, $P\colon \mathbb{R}^n \to C$ is also smooth.
In this setting, we may write out a first-order Taylor expansion for $Q_{S_k}(x) = \proj(x-\alpha\nabla F(x) - \alpha(\nabla F_{S_k}(x) - \nabla F(x)))$ as follows:
\begin{equation}
\label{eq:TaylorExpansion}
	\begin{aligned}	
	Q_{S_k}(x)
	&=
	Q(x)
	-
	\alpha\langle\nabla \!\proj(x-\alpha\nabla F(x)),\nabla F_{S_k}(x) - \nabla F(x)\rangle
	\\
	&\phantom{=} +
	\mcO(\alpha^2\|\nabla F_{S_k}(x) - \nabla F(x)\|^2)
	.
	\end{aligned}
\end{equation}
Therefore, because $\bbE_k[\nabla F_{S_k}(x) - \nabla F(x)] = 0$, by~\cref{eq:UnbiasedGradient}, we arrive at the second-order relationship
\begin{equation}
\label{eq:OrderEstimate}
	\|\bbE_k[Q_{S_k}(\varit) - Q(\varit)]\|
	=
	\mcO(\alpha^2\bbE_k[\|\nabla F_{S_k}(x_k) - \nabla F(x_k)\|^2])
	.
\end{equation}

If we recall~\cref{eq:SimpleIdentity}, it now seems appealing to replace \Cref{con:Bias} by an alternative condition that delivers a probabilistic threshold on $\nabla F_{S_k}(x_k)$ lying within a ball around $\nabla F(x_k)$:
\smallskip
\begin{condition}
\label{con:NEWcondition}
	Control of the error in the full gradient by the norm of the reduced gradient:
	\begin{equation}
	\label{eq:NEWcondition}
		\bbE_k[\|\nabla F_{S_k}(x_k) - \nabla F(x_k)\|^2]
		\leq
		\theta^2 \|R(x_k)\|^2
		,
	\end{equation}
	for some fixed $\theta > 0$.
\end{condition}
\smallskip 

\Cref{con:NEWcondition} is a direct generalization of the so-called ``norm test'' for stochastic gradient descent proposed in \cite{byrd2012sample}.
Although \Cref{con:NEWcondition} also requires unattainable foreknowledge of the exact gradient, it is possible to design a practical algorithm around it.
This aspect is discussed in the next section.
Before then, however, we establish a number of theoretical properties related to the conditions above.

We finish this subsection by showing that, under certain assumptions, \Cref{con:NEWcondition} implies \Cref{con:Norm,con:Bias}.
This observation is encapsulated in~\Cref{thm:Implication}.
The remainder of this section is devoting to analyzing the convergence of the stochastic projected gradient descent algorithm~\cref{eq:GeneralIteration} when either \Cref{con:Norm}, \ref{con:Bias}, or \ref{con:NEWcondition} is enforced.

\begin{theorem}
\label{thm:Implication}
	\Cref{con:NEWcondition} implies \Cref{con:Norm} with $\nu = \sqrt{2\theta + \theta^2}$.
	If, in addition,
	\begin{equation}
		\bbE[\|\nabla f(x;\xi) - \nabla F(x)\|^2] < \infty
	\label{eq:BoundedVarianceAssumption}
	\end{equation}
	for all $x\in C$, $\proj(\cdot):\R^n\to C$ is twice differentiable, and $|S_k|$ is sufficiently large, then \Cref{con:NEWcondition} implies \Cref{con:Bias} for some $\gamma \propto \theta$.
\end{theorem}

\begin{proof}
	To prove the first statement, it is important that we recall that $\proj(\cdot)$ is non-expansive~\cref{eq:non-expansive}.
	Due to this property, we have
	\begin{equation}
	\label{eq:ProjectionTrick}
	\begin{aligned}
		\|R_{S_k}(x_k) - R(x_k)\|
		&=
		\frac{1}{\alpha}
		\|\proj(x_k - \alpha\nabla F_{S_k}(x_k)) - \proj(x_k - \alpha\nabla F(x_k)))\|
		\\
		&\leq
		\|\nabla F_{S_k}(x_k) - \nabla F(x_k)\|
		.
	\end{aligned}
	\end{equation}
	Therefore, by~\cref{eq:NEWcondition},
	\begin{equation}
	\label{eq:LinearBoundBias}
		\|\bbE_k[R_{S_k}(x_k) - R(x_k)]\|
		\leq
		\big(\bbE_k[\|R_{S_k}(x_k) - R(x_k)\|^2]\big)^{1/2}
		\leq
		\theta\|R(x_k)\|
		.
	\end{equation}
	Likewise,
	\begin{align}
		\bbE_k[\|R_{S_k}(x_k)\|^2]
		&=
		\|R(x_k)\|^2
		+
		2\langle R(x_k), \bbE_k[R_{S_k}(x_k) - R(x_k)]\rangle
		+
		\bbE_k[\|R_{S_k}(x_k) - R(x_k)\|^2]
		\\
		\label{eq:ImplicationStep3}
		&\leq
		\|R(x_k)\|^2
		+
		2\|R(x_k)\| \|\bbE_k[R_{S_k}(x_k) - R(x_k)]\|
		+
		\bbE_k[\|R_{S_k}(x_k) - R(x_k)\|^2]
		\\
		&\leq
		\|R(x_k)\|^2
		+
		2\theta\|R(x_k)\|^2
		+
		\theta^2\|R(x_k)\|^2
		\\
		&\leq
		(1+\theta)^2\|R(x_k)\|^2
		.
	\end{align}
	In other words, \Cref{con:Norm} holds with $\nu = \sqrt{2\theta + \theta^2}$.

	To prove the second statement, we must argue that $\bbE_k[\|\nabla F_{S_k}(x_k) - \nabla F(x_k)\|^2] \to 0$ as $|S_k| \to \infty$.
	Indeed, notice that
	\begin{equation}
	\label{eq:NormLimit}
		\bbE_k[\|\nabla F_{S_k}(x_k) - \nabla F(x_k)\|^2]
		=
		\frac{\bbE_k[\|\nabla f(x_k;\xi) - \nabla F(x_k)\|^2]}{|S_k|}
		\to 0
		,
	\end{equation}
	since the numerator is independent of $|S_k|$ by~\cref{eq:BoundedVarianceAssumption}.
	Now, immediately following from~\cref{eq:OrderEstimate}, there exists some \revised{$\alpha$- and $k$-independent} constant, say $c\geq 0$, such that
	\begin{align}
		\|\bbE_k[Q_{S_k}(\varit) - Q(\varit)]\|
		&\leq
		c\alpha^2\bbE_k[\|\nabla F_{S_k}(x_k) - \nabla F(x_k)\|^2]
		,
	\end{align}
	for all sufficiently large $|S_k|$.
	Invoking \Cref{con:NEWcondition}, we have
	\begin{equation}
		\|\bbE_k[Q_{S_k}(\varit) - Q(\varit)]\|
		\leq
		c\theta^2\alpha^2\|R(x_k)\|^2
		=
		c\theta^2\|Q(x_k)-x_k\|^2
		,
	\end{equation}
	and thus \Cref{con:Bias} holds with $\gamma = \theta\sqrt{2c}$.
	This completes the proof.
\end{proof}

\begin{remark}
\label{rem:Affine}
	One may notice that if $C$ is \revisedBW{an affine subspace}, then the second-order term in~\cref{eq:TaylorExpansion} actually disappears and~\Cref{con:Bias} is satisfied trivially.
	We will argue in \Cref{sub:non_convex_problems} that this special setting permits us to propose other alternative conditions that are weaker than \Cref{con:NEWcondition}.
\end{remark}

\begin{remark}[Comparison to \cite{xie2020constrained}]
\label{rem:ComparisontoXie1}
	An alternative to \Cref{con:NEWcondition} which leads to similar convergence results is proposed in \cite[Equation~1.4]{xie2020constrained}.
	In our notation, this condition would be written
	\begin{equation}
	\label{eq:XieCondition}
		\bbE_k[\|\nabla F_{S_k}(x_k) - \nabla F(x_k)\|^2]
		\leq
		\theta^2 \|\bbE_k[R_{S_k}(x_k)]\|^2
		.
	\end{equation}
	It may be argued that the upper bound in~\cref{eq:XieCondition} is more expensive to estimate than $\|R(x_k)\|^2$ because a Monte Carlo estimate of $\bbE_k[R_{S_k}(x_k)]$ would require repeated application of the projection operator $P\colon \mathbb{R}^n\to C$.
	Meanwhile, estimating $\|R(x_k)\|^2$ requires only a careful estimate of $\nabla F(x_k)$ and a single application of $P$.
	\Cref{rem:ComparisontoXie2,rem:ComparisontoXie3} further compare our conditions to those in \cite{xie2020constrained}.
\end{remark}

\subsection{Convergence} %
\label{sub:convergence_rates}
Convergence of SPGD can be shown under a variety of assumptions involving \Cref{con:Norm,con:Bias,con:NEWcondition}.
We begin this subsection by showing that \Cref{con:NEWcondition} implies $q$-linear convergence when $F$ is strongly convex.

\begin{theorem}[Strongly convex objective]
\label{thm:StronglyConvex}
	Let $\Obj$ be both $L$-smooth~\cref{eq:LSmooth} and $\mu$-strongly convex~\cref{eq:MuStronglyConvex} and let $C$ be both convex and closed.
	Moreover, let the infinite sequence $\{\varit\}_{k=0}^\infty$ be generated by~\cref{eq:GeneralIteration}, with
	\begin{equation}
	\label{eq:StepSizeConditionStronglyConvex}
	 	\alpha
	 	<
	 	\frac{1}{L}
	\end{equation}
	and each $S_k$ satisfying \Cref{con:NEWcondition}.
	Then, for all sufficiently small $\theta>0$, $\varit$ converges $q$-linearly in expectation; i.e.,
	\begin{equation}
		\bbE[\lVert \varitnext - \varopt \rVert] \leq \rho^k \lVert x_0 - \varopt \rVert
		,
	\end{equation}
	for some $\rho \in [0,1)$, where $x^\ast = \argmin_{x\in C} F(x)$.
\end{theorem}

\begin{proof}
	By~\cref{eq:TotalExpectationIdentity}, it is sufficient to show that $\bbE_k[\lVert \varitnext - \varopt \rVert] \leq \rho \lVert \varit - \varopt \rVert$, for every $k$.
	To this end, denote $E_{S_k}(x_k) = Q_{S_k}(x_k) - Q(x_k)$ and observe that
	\begin{alignat}{3}
		\big(\bbE_k[\lVert \varitnext - \varopt \rVert]\big)^2
		& \leq
		\bbE_k[\lVert \varitnext - \varopt \rVert^2]
		= \bbE_k[\lVert \varit - \alpha\, \gradCsub(\varit) - \varopt\rVert^2] \\
		& = \lVert \varit - \varopt \rVert^2 + \alpha^2\,\bbE_k[\lVert \gradCsub(\varit) \rVert^2] - 2\alpha\,\bbE_k[\langle \gradCsub(\varit), \varit -\varopt\rangle] \\
		&=
		\lVert \varit - \varopt \rVert^2
		+
		\alpha^2\,\bbE_k[\lVert \gradCsub(\varit) \rVert^2]
		-
		2\alpha\,\langle\gradC(\varit), \varit -\varopt\rangle
		\\
		&\phantom{=} +
		2\,\langle \bbE_k[E_{S_k}(x_k)], \varit -\varopt\rangle
		.
	\end{alignat}
	Now, by \Cref{thm:Implication}, we have
	\begin{equation}
		\bbE_k[\lVert \gradCsub(\varit) \rVert^2]
		\leq
		(1+\nu^2)\,\lVert \gradC(\varit) \rVert^2
		,
	\end{equation}
	with $\nu^2 = 2\theta + \theta^2$.
	Furthermore, by~\cref{eq:StongConvexityIdentity}, we have
	\begin{equation}
		-2\alpha\langle\gradC(\varit), \varit -\varopt\rangle
		\leq
		-\mu\alpha \lVert \varit - \varopt \rVert^2
		-\alpha^2 \lVert \gradC(\varit) \rVert^2
		,
	\end{equation}
	and, by~\cref{eq:LinearBoundBias}, we have
	\begin{align}
		2\,\langle \bbE_k[E_{S_k}(x_k)], \varit -\varopt\rangle
		&\leq
		2\| \bbE_k[ E_{S_k}(x_k) ] \|\|\varit -\varopt\|
		\leq
		2\alpha\theta\| \gradC(\varit) \|\|\varit -\varopt\|
		.
	\end{align}
	Combining each of these bounds, we find that
	\begin{equation}
	\label{eq:StrongConvexityStep1}
		\bbE_k[\lVert \varitnext - \varopt \rVert]^2
		\leq
		\left( 1 - \mu\alpha\right)\lVert \varit - \varopt \rVert^2
		+
		2\alpha\theta\| \gradC(\varit) \|\|\varit -\varopt\|
		+
		\alpha^2(2\theta + \theta^2)\,\lVert \gradC(\varit) \rVert^2
		.
	\end{equation}
	
	Invoking~\cref{eq:StongConvexityIdentity} a second time, along with the Cauchy--Schwarz inequality, yields
	\begin{equation}
		\frac{\mu}{2} \lVert \varit - \varopt \rVert^2 + \frac{\alpha}{2} \lVert \gradC(\varit) \rVert^2
		\leq
		\|\gradC(\varit)\| \|\varit - \varopt\|
		.
	\end{equation}
	Note that $\mu\leq L$ and so $\alpha\mu < \mu/L \leq 1$.
	Moreover, the two roots of the equation $\mu a^2 + \alpha b^2 = 2ab$ are $b = (1 \pm \sqrt{1-\alpha\mu})a/\alpha$.
	Thus, it follows that
	\begin{equation}
	\label{eq:StrongConvexityStep2}
		(1 - \sqrt{1-\alpha\mu})\|x_k-x^\ast\|
		\leq
		\alpha\|\gradC(\var_k)\|
		\leq
		(1 + \sqrt{1-\alpha\mu})\|x_k-x^\ast\|
		.
	\end{equation}
	We may now replace every $\alpha\|\gradC(\var_k)\|$ factor in~\cref{eq:StrongConvexityStep1} by the upper bound given in~\cref{eq:StrongConvexityStep2}.
	A straightforward simplification of the resulting inequality yields
	\begin{equation}
		\bbE_k[\lVert \varitnext - \varopt \rVert]^2
		\leq
		\left( 1 + 2\big(1 + \sqrt{1-\alpha\mu}\big)(3\theta + \theta^2)  - (1 + \theta)^2 \mu\alpha \right)\lVert \varit - \varopt \rVert^2
		.
	\end{equation}
	Finally, note that if $\theta$ is chosen sufficiently small, then
	\begin{align}
		\rho^2
		&:=
		1 + 2\big(1 + \sqrt{1-\alpha\mu}\big)(3\theta + \theta^2)  - (1 + \theta)^2 \mu\alpha
		\leq
		1 + 4(3\theta + \theta^2)  - \mu\alpha
		<
		1
		,
	\end{align}
	as necessary.
\end{proof}

\Cref{con:Norm,con:Bias} can also be shown to imply convergence.
In the following theorem, we show that it is possible to arrive at a sublinear convergence rate with a general convex objective function $F$.

\begin{theorem}[General convex objective]
	Assume that $F$ is $L$-smooth~\cref{eq:LSmooth} and $C$ is convex and closed and that the sequence of iterates $\{\varit\}_{k=0}^\infty$ is contained in an bounded open set $D$ over which $\|\nabla F(x)\|$ is bounded above by some constant $M>0$.
	Moreover, assume that
	\begin{equation}
	\label{eq:GeneralConvexBound}
	 	\revised{
	 	\tilde{L}\alpha + \nu^2 + \gamma^2\mathrm{diam}(D) \leq 1
	 	,
	 	}
	\end{equation}
	where each $S_k$ satisfies \Cref{con:Norm,con:Bias}.
	Then, for every any positive integer $T$,
	\begin{equation}
	\E [\Obj(\var_{T})] - \Obj^*
	\leq
	\frac{1}{2\alpha T} \lVert x_0 - \varopt \rVert^2
	,
	\end{equation}
	where $\Obj^*$ is the optimal objective function value and $x^\ast \in \{x : x = \argmin_{x\in C} F(x) \}$.
\label{thm:GeneralConvex}
\end{theorem}

\begin{proof}
	Notice that
	\begin{equation}
		\|x_{k+1} - x_k\|^2 + \|x_k - x^\ast\|^2 - \|x_{k+1} - x^\ast\|^2
		=
		2\langle x_k - x_{k+1}, x_k - x^\ast\rangle.
	\end{equation}
	Using the identity $x_k - x_{k+1} = \alpha R_{S_k}(x_k)$ and rearranging terms, we arrive at
	\begin{equation}
		\|x_{k+1} - x^\ast\|^2 - \|x_{k} - x^\ast\|^2
		=
		2 \alpha \langle R(x_k), x^\ast - x_k\rangle
		-
		2 \langle E_{S_k}(x_k), x^\ast - x_k\rangle
		+
		\alpha^2 \|R_{S_k}(x_k)\|^2
		,
	\end{equation}
	where $E_{S_k}(x_k) = x_{k+1} - Q(x_k)$.
	Taking the expected value of both sides, we find
	\begin{align}
		\bbE_k[\|x_{k+1} - x^\ast\|^2] &- \|x_{k} - x^\ast\|^2
		\\
		&\leq
		2 \alpha \langle R(x_k), x^\ast - x_k\rangle
		+
		2\|\bbE_k[E_{S_k}(x_k)]\|\|x^\ast-x_k\|
		+
		\alpha^2 \bbE_k[\|R_{S_k}(x_k)\|^2]
		\\
		&\leq
		2 \alpha \langle R(x_k), x^\ast - x_k\rangle
		+
		\alpha^2(1+\nu^2+\gamma^2\|x^\ast-x_k\|)\|R(x_k)\|^2
		.
	\label{eq:WeakConvexityProof_eq1}
	\end{align}
	By setting $y=x^\ast$ and $z=x_k$ in \Cref{lem:GradientBounds}, we may write
	\begin{equation}
		\langle R(x_k), x^\ast - x_{k}\rangle
		+
		\langle R(x_k), x_k - Q(x_{k})\rangle
		\leq
		\langle \nabla F(x_k), x^\ast - x_{k}\rangle
		+
		\langle \nabla F(x_k), x_k - Q(x_{k})\rangle
		.
	\end{equation}
	Note that $\langle R(x_k), x_k - Q(x_{k})\rangle = \alpha\|R(x_k)\|^2$, by definition, and $\langle \nabla F(x_k), x^\ast - x_{k}\rangle  \leq  F^\ast - F(x_k)$, by convexity.
	By standard arguments following from the $L$-smoothness of $F$ \cite{bertsekas2015convex}, we have that
	\begin{align}
		\langle \nabla F(x_k), x_k &- Q(x_k) \rangle
		=
		\langle \nabla F(x_k), x_k - x_{k+1} \rangle
		+
		\langle \nabla F(x_k), x_{k+1} - Q(x_k) \rangle
		\\
		&\leq
		F(x_k)
		-
		F(x_{k+1})		
		+
		\frac{L}{2}\|x_{k+1} - x_k\|^2
		+
		\langle \nabla F(x_k), x_{k+1} - Q(x_k) \rangle
		.
	\end{align}
	Taking the conditional expectation of both sides yields
	\begin{align}
		\langle \nabla F(x_k), x_k &- Q(x_k) \rangle
		\\
		&\leq
		F(x_k)
		-
		\bbE_k[ F(x_{k+1}) ]		
		+
		\frac{L}{2} \bbE_k[\|x_{k+1} - x_k\|^2]
		+
		\langle \nabla F(x_k), \bbE_k[x_{k+1} - Q(x_k)] \rangle
		\\
		&\leq
		F(x_k)
		-
		\bbE_k[ F(x_{k+1}) ]
		+
		\frac{\alpha^2}{2} \big( L (1+\nu^2) + M \gamma^2 \big) \| R(x_k) \|^2
		.
	\end{align}
	Therefore,
	\begin{align}
		\langle R(x_k), x^\ast - x_{k}\rangle
		&\leq
		\langle \nabla F(x_k), x^\ast - x_{k}\rangle
		+
		\langle \nabla F(x_k), x_k - Q(x_{k})\rangle
		-
		\langle R(x_k), x_k - Q(x_{k})\rangle
		\\
		&\leq
		(F^\ast - F(x_k))
		+
		\Big(F(x_k)
		-
		\bbE_k[ F(x_{k+1}) ]
		+
		\frac{\tilde{L}}{2}\alpha^2\|R(x_k)\|^2
		\Big)
		-
		\alpha
		\|R(x_k)\|^2
		\\
		&=
		F^\ast - \bbE_k[ F(x_{k+1}) ]
		+
		\alpha\Big(\frac{\tilde{L}\alpha}{2}-1\Big)\|R(x_k)\|^2
		.
	\label{eq:WeakConvexityProof_eq2}
	\end{align}
	Finally, collecting together~\cref{eq:WeakConvexityProof_eq1,eq:WeakConvexityProof_eq2}, we find
	\begin{align}
		\bbE_k[\|x_{k+1} - x^\ast\|^2] &- \|x_{k} - x^\ast\|^2
		\\
		&\leq
		2\alpha(F^\ast - \bbE_k[ F(x_{k+1}) ] )
		+
		\alpha^2 ( \tilde{L} \alpha + \nu^2 + \gamma^2\|x^\ast-x_k\| -1) \|R(x_k)\|^2
		\\
		&\leq
		2\alpha(F^\ast - \bbE_k[ F(x_{k+1}) ] )
		,
	\end{align}
	where the second inequality follows from the bounds on $\alpha$, $\nu$, and $\gamma$ made in the theorem statement.
	We can now write
	\begin{equation}
		\bbE\big[F(x_{k+1})\big] - F^\ast
		\leq
		\frac{1}{2\alpha} \Big(
			\bbE[\|x_{k} - x^\ast\|^2] - \bbE[\|x_{k+1} - x^\ast\|^2]
		\Big)
		,
	\end{equation}
	which, after invoking~\Cref{thm:Descent}, delivers the bound
	\begin{align}
		\bbE\big[F(x_T)\big] - F^\ast
		&\leq
		\sum_{k=0}^{T-1}
		\frac{1}{T}
		(\bbE\big[F(x_{k+1})\big] - F^\ast)
		\\
		&\leq
		\frac{1}{2\alpha T}
		\Big(
		\bbE[\|x_{0} - x^\ast\|^2] - \bbE[\|x_{T} - x^\ast\|^2]
		\Big)
		\\
		&\leq
		\frac{1}{2\alpha T}
		\bbE[\|x_{0} - x^\ast\|^2]
		.
	\end{align}
	~

\end{proof}

\revised{
\begin{remark}
	The two previous theorems show $q$-linear and sublinear convergence rates, respectively.
	An important difference in \revisedBW{the} assumptions is that the sequence of iterates is assumed to be bounded in \Cref{thm:GeneralConvex}.
	Although this assumption trivially holds for any optimization problem posed over a bounded set, this does appear to be a strong assumption to make.
	At the same time, if one analyzes the proof in detail, it can be seen \revisedBW{that $\tilde{L}\alpha + \nu^2 + \gamma^2\|x^\ast-x_k\| \leq 1$ for each $x_k$ is sufficient}.
	This fact suggests the possible benefit of evolving $\gamma>0$ with $x_k$.
	It also demonstrates the deteriorating role of the bias (cf.~\cref{eq:ExactBiasControl}) on the expected accuracy as $x_k$ approaches $x^\ast$.
\end{remark}
}

The following theorem shows an even weaker version of convergence which requires only the same mild assumptions as were made in \Cref{thm:Descent}.
In particular, it shows that the sequence of reduced gradients $\{R(x_k)\}$ converges to
zero in expectation.
Therefore, every limit point $x^\ast$ of the sequence $x_k$ is stationary; i.e., $Q(x^\ast) = x^\ast$.
This theorem also establishes a global sublinear rate of convergence of the smallest reduced gradients.

\begin{theorem}[Non-convex objective]
\label{thm:non-convex_objective}
	Under the assumptions of~\Cref{thm:Descent}, if $\alpha < 2/\tilde{L}$, then it holds that
	\begin{equation}
		\lim_{k\to\infty}
		\bbE[\|R(x_k)\|^2]
		=
		\lim_{k\to\infty}
		\bbE[\|Q(x_k) - x_k\|^2]
		=
		0
		.
	\end{equation}
	Moreover, for any positive integer $T$,
	\begin{equation}
		\min_{0\leq k \leq T-1} \bbE[\|R(x_k)\|^2]
		\leq
		\frac{1}{c\alpha^2T}\big(F(x_{0}) - F_{\min}\big)
		,
	\end{equation}
	where $c = \frac{1}{\alpha} - \frac{\tilde{L}}{2} > 0$ and $F_{\min}$ is a finite lower bound on $F$ in $C$.
\end{theorem}

\begin{proof}
	Begin by taking the total expected value of both sides of~\cref{eq:DescentTheorem_eq5} and rewriting the result as
	\begin{equation}
		\bbE[\|R(x_k)\|^2]
		=
		\frac{1}{\alpha^2}\bbE[\|Q(x_k)-x_k\|^2]
		\leq
		\frac{1}{c\alpha^2}\big(\bbE[F(x_{k})] - \bbE[F(x_{k+1})]\big)
		.
	\label{eq:NonConvexProofStep1}
	\end{equation}
	It follows from the step size assumption in~\Cref{thm:Descent} that $c > 0$.
	Therefore, summing both sides of~\cref{eq:NonConvexProofStep1} delivers
	\begin{equation}
		\sum_{k=0}^{T-1}
		\bbE[\|R(x_k)\|^2]
		\leq
		\frac{1}{c\alpha^2}\big(\bbE[F(x_{0})] - \bbE[F(x_{T})]\big)
		\leq
		\frac{1}{c\alpha^2}\big(F(x_{0}) - F_{\min}\big)
		.
	\end{equation}
	Since this sum of $T$ positive terms is bounded from above by a constant independent of $T$, the first statement follows.
	Moreover, notice that
	\begin{equation}
		\min_{0\leq k \leq T-1} \bbE[\|R(x_k)\|^2]
		\leq
		\frac{1}{T}
		\sum_{k=0}^{T-1}
		\bbE[\|R(x_k)\|^2]
		\leq
		\frac{1}{c\alpha^2T}\big(F(x_{0}) - F_{\min}\big)
		.
	\end{equation}
	This completes the proof.
\end{proof}

\begin{remark}[Comparison to \cite{xie2020constrained}]
\label{rem:ComparisontoXie2}
	In \cite[Theorem 3.3]{xie2020constrained}, it is shown that~\cref{eq:XieCondition} also leads to q-linearly convergence in expectation when $F$ is strongly convex and sublinear convergence when $F$ is convex, but not strongly convex.
	No theorem similar to \Cref{thm:non-convex_objective}, for convergence in the case of non-convex $F$, appears in \cite{xie2020constrained}.
\end{remark}

\section{A practical algorithm} %
\label{sub:a_practical_algorithm}

In this section, we develop a practical SPGD algorithm based on \Cref{con:NEWcondition}.
In order to test whether this condition is satisfied, we introduce an approximation of the true gradient $\nabla F(x_k)$ and the risk measure $\bbE[\cdot]$.
We begin by recalling~\cref{eq:NormLimit}, which allows us to we rewrite~\cref{eq:NEWcondition} as
\begin{equation}
\label{eq:Condition3Rewrite}
	\frac{\bbE_k[\|\nabla f(x_k;\xi) - \nabla F(x_k)\|^2]}{|S_k|}
	\leq
	\theta^2 \|R(x_k)\|^2
	.
\end{equation}
We then approximate the true gradient $\nabla F(x_k)$ by the sample average gradient $\nabla F_{S_k}(x_k)$, as done in similar work on adaptive sampling; cf. \cite{Bollapragada2019}.
Likewise, we approximate the conditional expected value $\bbE_k[\cdot]$ by a sample average.
Altogether, we propose the following practical test to check \Cref{con:NEWcondition}:
\smallskip
\begin{test}[Approximation of \Cref{con:NEWcondition}]
\label{con:NewCondition_approx}
	Approximate control of the error in the full gradient by the norm of the reduced gradient:
	\begin{equation}
	\label{eq:NEWcondition_inexact}
		\frac{1}{|S_k|-1}\frac{\sum_{\xi\in S_k}\|\nabla f(x_k;\xi) - \nabla F_{S_k}(x_k)\|^2}{|S_k|}
		\leq
		\theta^2 \|R_{S_k}(x_k)\|^2
		,
	\end{equation}
	for some fixed $\theta > 0$.
\end{test}
\smallskip
In~\cref{eq:NEWcondition_inexact}, we have used the factor $\frac{1}{|S_k|-1}$ instead of $\frac{1}{|S_k|}$ so that the left-hand side becomes an unbiased estimator for $\bbE_k[\|\nabla F_{S_k}(x_k) - \nabla F(x_k)\|^2]$.

In order to construct a set $S_k$ satisfying~\cref{eq:NEWcondition_inexact}, one may envision starting with a sample set $S_k$ of a minimal size, say $|S_k| = |S_0|$, and simply adding samples until~\cref{eq:NEWcondition_inexact} holds.
This strategy, however, would be too expensive to be practical as it would require recomputing $R_{S_k}(x_k)$ each time the set $S_k$ is updated.
Because of the expense of applying $\proj(\cdot)$, we choose to only consider strategies which involve computing $R_{S_k}(x_k)$ once each iteration.

One natural thing to consider is to use \cref{eq:NEWcondition_inexact} to predict the correct size of the \emph{upcoming} sample set $S_{k+1}$.
The prediction of an a posteriori sample size for the next iteration is also presented in \cite{Bollapragada2018a}, where unconstrained problems are considered.
For the constrained optimization problems at hand, such a strategy may work as follows.
Begin by dividing the left-hand side of~\cref{eq:NEWcondition_inexact} by $\theta^2 \lVert \gradCsub(\varit) \rVert^2$ and, in turn, define the new quantity
\begin{equation}
	\rho
	=
	\frac{\sum_{\xi\in S_k}\|\nabla f(x_k;\xi) - \nabla F_{S_k}(x_k)\|^2}{\theta^2 (|S_k|-1)|S_k|\lVert \gradCsub(\varit) \rVert^2}
	.
\label{eq:InexactTests}
\end{equation}
When~\cref{eq:NEWcondition_inexact} is satisfied, we clearly have $\rho \leq 1$, and we simply keep the sample size fixed; that is, $|S_{k+1}| = |S_k|$.
On the other hand, if the test fails, $\rho>1$ is used to increase the sample size via the update rule
\begin{equation}
	|S_{k+1}| = \lceil\rho\,|S_k|\rceil
	.
\label{eq:UpdateRule}
\end{equation}
The procedure above leads to the following algorithm:

\begin{algorithm2e}[H]
\DontPrintSemicolon
	\caption{\label{alg:practical} SPGD adaptive sampling algorithm for convex stochastic programs}
	\SetKwInOut{Input}{input}
	\Input{$\var_0$, step size $\alpha>0$, initial sample set $S_0$, sampling rate parameter $\theta > 0$}
	Set $\its \leftarrow 0$.\;
	\Repeat{a convergence test is satisfied}
	{
		Update $\varitnext = Q_\subsampling(\varit)$.\;
		\If{\Cref{con:NewCondition_approx} is not satisfied}
		{Construct $S_{k+1}$ obeying~\cref{eq:UpdateRule}.\;}
		\Else
		{Construct $S_{k+1}$ satisfying $|S_{k+1}| = |S_k|$.\;}
		Set $\its \leftarrow \its+1$.\;
	}
\end{algorithm2e}

\revised{
\begin{remark}
	Upon rewriting $R_{S_k}(x_k) = (x_k - x_{k+1})/\alpha$, it is clear that \Cref{con:NewCondition_approx} can be checked without performing any additional projections.
	The fact that~\Cref{alg:practical} requires only one projection per iteration makes the algorithm particular appealing when the most expensive ingredient is the evaluation of the projection operator $P\colon \R^n \to C$.
\end{remark}
}

\begin{remark}
	\revised{
	It is important to note that~\Cref{alg:practical} is only an approximation of the algorithm analyzed in~\Cref{thm:StronglyConvex} since \Cref{con:NewCondition_approx} differs in two important ways from \Cref{con:NEWcondition}.
	In particular, the reduced gradient $R(x_k)$ on the right-hand side of~\Cref{con:NEWcondition} is replaced by the estimate $R_{S_k}(x_k)$ and the variance on the left-hand side of~\cref{eq:Condition3Rewrite} is replaced by the sample variance $\frac{1}{|S_k|-1}\sum_{\xi\in S_k}\|\nabla f(x_k;\xi) - \nabla F_{S_k}(x_k)\|^2$.
	Because of these differences,} we cannot guarantee the same convergence rates predicted by \Cref{thm:StronglyConvex}.
	Nevertheless, as we will see in~\Cref{sec:applications}, our experiments with \Cref{alg:practical} demonstrate extremely good agreement with the theoretical results of \Cref{thm:StronglyConvex}.
	Previous authors have made similar observations for their own practical adaptive sampling strategies \cite{Bollapragada2018a,xie2020constrained}.
	These repeated observations hint at a promising robustness in the adaptive sampling technique used here.
\end{remark}

\begin{remark}[Comparison to \cite{xie2020constrained}]
\label{rem:ComparisontoXie3}
	Although the original conditions and analysis differ in numerous ways, the practical adaptive sampling algorithm proposed in \cite[Section~3.5]{xie2020constrained} differs only marginally from~\Cref{alg:practical}.
	Indeed, the only minor difference is that the practical algorithm in \cite{xie2020constrained} requires computing a second search direction before advancing to the next iteration when \Cref{con:NewCondition_approx} is not satisfied.
	\revised{We propose a variant of this approach in~\Cref{sec:nonconvex-applications}.}
\end{remark}

\section{Risk-averse problems} %
\label{sec:risk_averse_problems}

\revised{We now turn towards extending the algorithm proposed above}
to stochastic programs involving the conditional value-at-risk.
We present two different approaches \revised{to achieving this goal}; both involve a regularization technique proposed in \cite{kouri2016risk} and rewriting $\CVaR_\beta(X)$ as the solution of an auxiliary optimization problem\revised{.}
Our first method follows a well-established course of action in risk-averse stochastic programming \cite{kouri2016risk,Shapiro2009,Kouri2018a} and conforms to the assumptions used in the previous \revised{section}.
Our second method involves solving an additional one-dimensional optimization problem at each iteration.

\subsection{Conditional value-at-risk}
\label{sub:cvar}

Let $\Psi_X(x) := \bbP(X\leq x)$ denote the cumulative distribution function (CDF) of a random variable $X$.
The value-at-risk ($\mathrm{VaR}$) of $X$, at confidence level $0<\beta<1$, also known as the $\beta$-quantile, is defined by
\begin{equation}
\mathrm{VaR}_\beta(X)
:=
\inf\,\{t\in\R \,:\, \Psi_X(t) \geq \beta\}
\,.
\label{eq:VaR}
\end{equation}
The conditional value-at-risk ($\mathrm{CVaR}$) of $X$, at confidence level $\beta$, is essentially the expected value of $X$ beyond $\mathrm{VaR}_\beta(X)$.
Indeed, if $\Psi_X(x)$ is right-continuous, then $\mathrm{CVaR}_\beta(X)$ is precisely the conditional expectation $\bbE[X| X>\mathrm{VaR}_\beta(X)]$.
This implies that $\mathrm{CVaR}_\beta(X)\geq\E[X]$.
In order to accommodate more general CDFs, one may alternatively define $\mathrm{CVaR}_\beta(X)$ as the weighted integral of the value-at-risk over the interval $(\beta,1)$,
\begin{align}
\mathrm{CVaR}_\beta(X)
&:=
\frac{1}{1-\beta}
\int_\beta^1
\mathrm{VaR}_\alpha(X)
\dd\alpha.
\label{eq:CVaR-def}
\end{align}
Since $\mathrm{VaR}_\alpha(X)$ is a non-decreasing function of $\alpha$, note that
\begin{equation}
\mathrm{CVaR}_\beta(X)
\geq
\frac{1}{1-\beta}
\mathrm{VaR}_\beta(X)
\int_\beta^1
\dd\alpha
=
\mathrm{VaR}_\beta(X)
\,.
\label{eq:CVaRBoundsVaR}
\end{equation}

In many applications, $\mathrm{CVaR}_\beta(X)$ is a more useful measure of risk than $\mathrm{VaR}_\beta(X)$ because controlling expected failure states \revised{can be} more important than controlling the most optimistic failure state \revised{only}.
For instance, consider
\revised{the case where}
$X$ can be identified with a stress acting on/within a physical system.
In such scenarios, lower values of $X$ are generally preferable to higher values of $X$.
Thus, $\mathrm{VaR}_\beta(X)$ represents the most optimistic value that $X$ can achieve in the worst $(1-\beta)\cdot 100$ percent of possible events.
Alternatively, $\mathrm{CVaR}_\beta(X)$ represents the expected value of $X$ in the worst $(1-\beta)\cdot 100$ percent of possible events.

The properties above make $\CVaR_\beta(X)$ a suitable risk measure for industrial optimization problems \cite{TyrrellRockafellar2015}.
There are a variety of ways to treat stochastic programs which incorporate the $\CVaR$ \cite{kouri2016risk,Curi2019}.
However, in this work, we find that the following ``dual formulation'' is particularly useful.

In \cite{Rockafellar2000} it is shown that $\CVaR_\beta(X)$ can be interpreted as the solution of a scalar optimization problem; namely,
\begin{align}
	\mathrm{CVaR}_\beta(X)
	=
	\inf_{t\in\R}
	\Big\{
	t + \frac{1}{1-\beta}\bbE[(X-t)_+]
	\Big\}
	\,,
	\label{eq:CVaRdual}
\end{align}
where $(x)_+ := \max\{0,x\}$.
Therefore, the stochastic program
\begin{equation}
	\min_{\var\in C}~
	F(x) = 
	\CVaR_\beta[\obj(\var;\rv)]
\label{eq:CVaR_opt}
\end{equation}
can be conveniently reformulated as 
\begin{equation}
	\min_{(\var,t)\in C\times\R}~
	F(x,t) = \E\Big[t + \frac{1}{1-\beta}(\obj(\var;\rv)-t)_+\Big]
	.
\label{eq:CVaR_extended}
\end{equation}
It is well-known that non-smoothness of the operator $(\,\cdot\,)_+$, implies non-smoothness of the objective function $F(x,t)$ \cite{Rockafellar2000}.
Therefore,~\cref{eq:CVaR_extended} is often solved with subgradient types methods; see, e.g., \cite{Royset2013}.
An alternative option is to replace $\CVaR_\beta$ by a smooth approximation, which maintains many of its essential properties.
In this work, we choose to use a smoothing technique proposed by Kouri and Surowiec \cite{kouri2016risk}.

\subsection{Smoothing}
\label{sub:smoothing}

The non-differentiability of $F(x,t)$ can be circumvented by regularizing the $(\cdot)_+$ function.
In \cite[Section~4.1.1.]{kouri2016risk}, several strategies are proposed.
We choose the smooth approximation $(\cdot)_+^\varepsilon$ defined as follows:
\begin{align}
(y)_+^\varepsilon = y + \varepsilon \ln \left(1 + \exp \left(\frac{-y}{\varepsilon}\right)\right).
\end{align}
Likewise, we replace the non-smooth $\CVaR_\beta$ risk measure by the smoothed risk measure
\begin{align}
	\text{CVaR}^\varepsilon_{\beta}(X)
	=\inf_{t\in\R}
	\Big\{
	t + \frac{1}{1-\beta}\bbE[(X-t)_+^\varepsilon]
	\Big\}
\label{eq:smooth_cvar}
\end{align}
and replace~\cref{eq:CVaR_extended} by 
\begin{equation}
	\min_{(\var,t)\in C\times\R}~
	F^\varepsilon(x,t) = \E\Big[t + \frac{1}{1-\beta}(\obj(\var;\rv)-t)_+^\varepsilon\Big]
	.
\label{eq:CVaR_extended_smooth}
\end{equation}

All of the conclusions in the previous sections carry over to the regularized CVaR problem because the objective function $F^\varepsilon(x,t)$ is now smooth.
This means that~\Cref{alg:practical} can be used to solve~\cref{eq:CVaR_extended_smooth}.
It is also important to point out that this smooth CVaR formulation enjoys the advantage that many of the original CVaR properties are preserved, including convexity and monotonicity \cite{kouri2016risk}.
Accordingly, if $\obj(\var;\rv)$ is convex for almost every $\rv$, then $F^\varepsilon(x,t)$ is also convex.

\begin{remark}
\label{rem:bound}
The regularization constant $\varepsilon$ is a problem-dependent parameter which must be tuned.
To guide the tuning process, one may use Lemma~4.3 in \cite{kouri2016risk}, which shows that
\begin{equation}
	| \text{CVaR}^\varepsilon_{\beta}(X) - \text{CVaR}_{\beta}(X) |
	\leq
	\frac{\log 2}{1-\beta} \: \varepsilon
	.
\end{equation}
Thus, the value of $\varepsilon$ necessary to achieve an intended relative error will depend on both the magnitude of $\text{CVaR}_{\beta}(X)$ and the confidence level $\beta$.
A short study on the influence of $\varepsilon$ is carried out in \cite[Chapter~5.1.3]{Urbainczyk2020}.
\end{remark}

\begin{remark}
	The function $(\cdot)_+$ falls into \revised{the} special class of so-called ``scalar regret functions'' \cite{rockafellar2013fundamental}. Specifically, \revised{such} functions $v:\R \to \overline{\R}$ are closed, convex, increasing, and satisfy $v(0) = 0$ and $v(x) > x$ for all $x\neq 0$.
	If one replaces $\frac{1}{1-\beta}(\cdot)_+^\varepsilon$ in~\cref{eq:CVaR_extended_smooth}, with any scalar regret function $v(\cdot)$, then one arrives at an important class of risk-averse stochastic programs, which has also received a great deal of attention \cite{ben1986expected,rockafellar2013fundamental,Kouri2018a}:
	\begin{equation}
		\min_{\var\in C}~
		F(\var) = \mathcal{R}[\obj(\var;\rv)]
		\quad
		\text{where}
		\quad
		\mcR(X)
		=\inf_{t\in\R}
		\Big\{
		t + \bbE[v(X-t)]
		\Big\}
		.
	\end{equation}
	If $v$ is also smooth, then~\Cref{alg:practical} may also be used without further modification to solve this entire family of risk-averse stochastic programs.
\end{remark}

\subsection{Nested quantile estimation} %
\label{sub:alternative_algorithm}

Although~\Cref{alg:practical} can be used to solve~\cref{eq:CVaR_extended_smooth}, when there are only a small number of samples, the initial error may be quite large; cf. \Cref{ssub:risk_averse_portfolio_optimization}.
For this reason, we introduce an alternative algorithm.
We begin with two observations.

It is well-known that the unique minimizer of~\cref{eq:CVaRdual}, $t^\ast$, is simply the value-at-risk; namely,
\begin{equation}
	\mathrm{CVaR}_\beta(X)
	=
	t^\ast + \frac{1}{1-\beta}\bbE[(X-t^\ast)_+]
	,
	\quad
	\text{where}
	\quad
	t^\ast
	=
	\mathrm{VaR}_\beta(X)
	.
\end{equation}
Accordingly, if we assume that $\mathrm{VaR}_\beta(\obj(\var;\rv))$ was somehow determined \textit{a priori}, it would be possible to rewrite~\cref{eq:CVaR_opt} as
\begin{equation}
	\min_{\var\in C}~
	\tilde{F}(x)
	= 
	\bbE[(\obj(\var;\rv)-\mathrm{VaR}_\beta(\obj(\var;\rv)))_+]
	.
\label{eq:CVaR_alt}
\end{equation}
This technique of rewriting~\cref{eq:CVaR_opt} is analogous to the scalar regret function reformulation of stochastic programs involving the entropic risk measure; see, e.g., \cite[Section~2.4.2]{Kouri2018a}.

It turns out that there are a large number of methods to estimate quantiles which are widely available in scientific software such as R \cite{Rmanual}, Python (specifically, SciPy \cite{2020SciPy}), and Julia \cite{bezanson2017julia}.
Any of these approximations could be substituted for $\mathrm{VaR}_\beta(\obj(\var;\rv))$ in~\cref{eq:CVaR_alt}, once a set of samples of $\obj(\var;\rv)$ is collected.
Nevertheless, we choose to approximate the value-at-risk by estimating $t^\ast$ at each iteration and then solving the regularized form of~\cref{eq:CVaR_alt}.
That is, we first compute
\begin{equation}
\label{eq:optimal_t}
	t_{S_k} = \argmin_{t\in\mathbb{R}} \: \left\{t + \frac{1}{1-\beta} \:\frac{1}{|S_k|}\: \sum_{\rv_i\in S_k}(\obj(\var; \rv_i) - t)_+^\varepsilon\right\}
\end{equation}
with a root finding algorithm.
This is no more expensive that a standard line search and generally cheaper than applying $\proj(\cdot)$.
Furthermore, one may argue that $t_{S_k} \to t^\ast$ as $|S_k|\to \infty$.
We then compute the new iterate $x_{k+1} = Q_{S_k}(x_k)$ via the subsampled gradient map of
\begin{equation}
	\tilde{F}^\varepsilon_{S_k}(x)
	:=
	\frac{1}{|S_k|}\sum_{\rv_i\in S_k}(\obj(\var;\rv)-t_{S_k})_+^\varepsilon
	.
\end{equation}
The entire adaptive sampling process is described in \Cref{alg:alternative_cvar}, below.

\begin{algorithm2e}[H]
\DontPrintSemicolon
	\caption{\label{alg:alternative_cvar}Nested quantile estimation and adaptive sampling with $\CVaR$}
	\SetKwInOut{Input}{input}
	\Input{$\var_0$, step size $\alpha>0$, initial sample set $S_0$, constant $\theta > 0$}
	Set $\its \leftarrow 0$.\;
	\Repeat{a convergence test is satisfied}
	{
		Compute $t_{S_k} = \argmin_{t\in\mathbb{R}} \: \big\{t + \frac{1}{1-\beta} \:\frac{1}{|S_k|}\: \sum_{\rv_i\in S_k}(\obj(\varit; \rv_i) - t)_+^\varepsilon\big\}$.\;
		Update $\varitnext = \argmin_{y\in C}
		\big\{
		\tilde{F}_\subsampling(\varit) + \langle \nabla \tilde{F}_\subsampling(\varit),y-\varit \rangle + \frac{1}{2\alpha}\lVert y-\varit \rVert^2
		\big\}$.\;
		\If{\Cref{con:NewCondition_approx} is not satisfied}
		{Construct $S_{k+1}$ obeying~\cref{eq:UpdateRule}.\;}
		\Else
		{Construct $S_{k+1}$ satisfying $|S_{k+1}| = |S_k|$.\;}
		Set $\its \leftarrow \its+1$.\;
	}
\end{algorithm2e}

\section{Numerical examples} %
\label{sec:applications}

In this section, we conduct
\revised{two sets of numerical experiments to illustrate \Cref{alg:practical,alg:alternative_cvar}.}
We begin with a simple example problem which allows us to test the theory presented in \Cref{sub:convergence_rates}.
Subsequently, we assess the practicality and robustness of the adaptive sampling algorithms with a risk-averse portfolio optimization application.
In order to discuss the performance of \Cref{alg:practical,alg:alternative_cvar}, we include plots showing the objective function values at each iteration. These function values were estimated to a high accuracy independent of the algorithms' approximation of the objective function value.

\subsection{Basic example}
\label{subsec:basic_example}
Our first stochastic programming example is inspired by \cite[Section~6.2]{Royset2013}.
Consider a function
\begin{equation}
	\obj(\var;\rv) = \sum_{l=1}^{20}{a_l(\var^l-b_l\rv^l)^2}
	,
	\qquad
	x = (x^1,\ldots,x^{20}),
	\quad
	\xi = (\xi^1,\ldots,\xi^{20}),
\label{eq:Basic_f}
\end{equation}
where the coefficients $a_l \sim \mathsf{Unif}(1,2)$ and $b_l \sim \mathsf{Unif}(-1,1)$ have been randomly sampled once for the sake of simulation and, thereafter, left fixed.
Next, assume that $\rv$ is a random vector where each coefficient $\rv^l \sim \mathsf{Unif}(0,1)$.
Finally, define the admissible set $C=[0,\infty)^{20}$, which is closed, convex, and unbounded.

With the definitions given above, we consider the (risk-neutral) stochastic program  
\begin{equation}
	\min_{\var\in C} ~ \Big\{ \Obj(\var) = \E[\obj(\var;\rv)] \Big\}.
	\label{eq:basic_example}
\end{equation}
Note that this program is strongly convex and that $\obj(\var;\rv)$ is differentiable for every $\rv\in\Xi = [0,1]^{20}$.
Therefore, there exists a unique global minimizer and \Cref{thm:StronglyConvex} applies.
In fact, the unique global minimizer $\varopt=(x^{\ast,1},\dots,x^{\ast,20})$ of~\cref{eq:basic_example} can be written out explicitly; i.e.,
$x^{\ast,l} = \max\{0,{b_l}/{2}\}$, for each $l=1,\dots,20$.

This example has two purposes: first, to suggest that the theory presented in \Cref{sub:convergence_rates} also holds when the practical \Cref{con:NewCondition_approx} is used and, second, to compare the performance of \Cref{alg:practical} with different values of $\theta$.
In \Cref{fig:basic} we see the results from six representative optimization runs.
The first three runs use fixed sample sizes of $|S_k| = 10,~10^3$, and $10^5$, respectively, for all iterations $k$; these runs imitate naive approaches to compare against.
The subsequent three runs each begin with the common initial sample size $|S_0|=10$ and are executed using \Cref{alg:practical} with the parameter values $\theta = 0.5,~1.0$, and $1.5$, respectively.
All of the runs use a fixed step size of $\alpha = 0.025$.
Due to~\Cref{thm:StronglyConvex}, similar results are expected for all step sizes $\alpha < 1/L$ and sufficiently small $\theta > 0$.
We present further experiments on the influence of the step size in \Cref{app:stepsize}, \revised{whereas in the current section we focus on the effects of the adaptive sampling}.

\begin{figure}
	\centering
	\begin{minipage}{0.4\textwidth}
		\centering
		\begin{scaletikzpicturetowidth}{\textwidth}
			\begin{tikzpicture}[scale=\tikzscale,font=\large]
\begin{axis}[
xlabel={Iteration},
ylabel={$\lVert x^* - x_k \rVert$},
xmajorgrids,
ymajorgrids,
ymode = log,
legend pos=south west,
legend style={fill=white, fill opacity=0.6, draw opacity=1, text opacity=1, font=\normalsize},
]

\addplot[black!75!white, ultra thick, dashed, restrict x to domain=0:150] table [x index={0}, y index={3}, col sep=comma] {./Results/BasicExample/statistics_fix10e1.csv};
\addlegendentry{$|S|=10^1$};

\addplot[color1!75!white, ultra thick, dashed, restrict x to domain=0:150] table [x index={0}, y index={3}, col sep=comma] {./Results/BasicExample/statistics_fix10e3.csv};
\addlegendentry{$|S|=10^3$};

\addplot[color2!75!white, ultra thick, dashed, restrict x to domain=0:150] table [x index={0}, y index={3}, col sep=comma] {./Results/BasicExample/statistics_fix10e5.csv};
\addlegendentry{$|S|=10^5$};

\addplot[color3, ultra thick, restrict x to domain=0:150] table [x index={0}, y index={3}, col sep=comma] {./Results/BasicExample/statistics_050.csv};
\addlegendentry{$\theta=0.5$};

\addplot[color5, ultra thick, restrict x to domain=0:150] table [x index={0}, y index={3}, col sep=comma] {./Results/BasicExample/statistics_100.csv};
\addlegendentry{$\theta=1.0$};

\addplot[color4, ultra thick, restrict x to domain=0:150] table [x index={0}, y index={3}, col sep=comma] {./Results/BasicExample/statistics_150.csv};
\addlegendentry{$\theta=1.5$};

\end{axis}
\end{tikzpicture}
		\end{scaletikzpicturetowidth}
	\end{minipage}%
	~~
	\begin{minipage}{0.4\textwidth}
		\centering
		\begin{scaletikzpicturetowidth}{\textwidth}
			\begin{tikzpicture}[scale=\tikzscale,font=\large]
\begin{semilogyaxis}[
xlabel={Iteration},
ylabel={Sample size},
xmajorgrids,
ymajorgrids,
ymode = log,
legend style={fill=white, fill opacity=0.6, draw opacity=1, text opacity=1},
legend pos=north west
]

\addplot[black!75!white, ultra thick, dashed, domain=0:99] {10};
\addlegendentry{$|S| = 10$};

\addplot[color1!75!white, ultra thick, dashed, domain=0:99] {1000};
\addlegendentry{$|S| = 10^3$};

\addplot[color2!75!white, ultra thick, dashed, domain=0:99] {100000};
\addlegendentry{$|S| = 10^5$};

\addplot[color3, ultra thick, restrict x to domain=0:150] table [x index={0}, y index={1}, col sep=comma] {./Results/BasicExample/statistics_050.csv};
\addlegendentry{$\theta = 0.5$};

\addplot[color5, ultra thick, restrict x to domain=0:150] table [x index={0}, y index={1}, col sep=comma] {./Results/BasicExample/statistics_100.csv};
\addlegendentry{$\theta = 1.0$};

\addplot[color4, ultra thick, restrict x to domain=0:150] table [x index={0}, y index={1}, col sep=comma] {./Results/BasicExample/statistics_150.csv};
\addlegendentry{$\theta = 1.5$};

\legend{}; %
\end{semilogyaxis}
\end{tikzpicture}
		\end{scaletikzpicturetowidth}
	\end{minipage}%
	\\
	\begin{minipage}{0.4\textwidth}
		\centering
		\begin{scaletikzpicturetowidth}{\textwidth}
			\begin{tikzpicture}[scale=\tikzscale,font=\large]
\begin{axis}[
xlabel={Iteration},
ylabel={$F(x_k)-F^*$},
xmajorgrids,
ymajorgrids,
legend cell align={left},
ymode = log,
legend style={fill=white, fill opacity=0.6, draw opacity=1, text opacity=1},
legend pos=north east,
declare function = {F_opt=1.547467255164272;},
]

\addplot[black!75!white, ultra thick, dashed, restrict x to domain=0:150] table [x index={0}, y index={6}, col sep=comma] {./Results/BasicExample/statistics_fix10e1.csv};
\addlegendentry{SAA, $|S| = 10^1$};

\addplot[color1!75!white, ultra thick, dashed, restrict x to domain=0:150] table [x index={0}, y index={6}, col sep=comma] {./Results/BasicExample/statistics_fix10e3.csv};
\addlegendentry{SAA, $|S| = 10^3$};

\addplot[color2!75!white, ultra thick, dashed, restrict x to domain=0:150] table [x index={0}, y index={6}, col sep=comma] {./Results/BasicExample/statistics_fix10e5.csv};
\addlegendentry{SAA, $|S| = 10^5$};

\addplot[color3, ultra thick, restrict x to domain=0:150] table [x index={0}, y index={6}, col sep=comma] {./Results/BasicExample/statistics_050.csv};
\addlegendentry{$\theta = 0.5$};

\addplot[color5, ultra thick, restrict x to domain=0:150] table [x index={0}, y index={6}, col sep=comma] {./Results/BasicExample/statistics_100.csv};
\addlegendentry{$\theta = 1.0$};

\addplot[color4, ultra thick, restrict x to domain=0:150] table [x index={0}, y index={6}, col sep=comma] {./Results/BasicExample/statistics_150.csv};
\addlegendentry{$\theta = 1.5$};

\legend{}; %
\end{axis}
\end{tikzpicture}

		\end{scaletikzpicturetowidth}
	\end{minipage}%
	~~
	\begin{minipage}{0.4\textwidth}
		\centering
		\begin{scaletikzpicturetowidth}{\textwidth}
			\begin{tikzpicture}[scale=\tikzscale,font=\large]
\begin{axis}[
xlabel={Gradient evaluations},
ylabel={$F(x_k)-F^*$},
xmajorgrids,
ymajorgrids,
xmode = log,
ymode = log,
legend style={fill=white, fill opacity=0.6, draw opacity=1, text opacity=1},
legend pos=north east,
]

\addplot[black!75!white, ultra thick, dotted, restrict x to domain=0:1000000] table [x index={2}, y index={6}, col sep=comma] {./Results/BasicExample/statistics_fix10e1.csv};
\addlegendentry{SAA, $|S| = 10$};

\addplot[color1!75!white, ultra thick, dotted, restrict x to domain=0:1000000] table [x index={2}, y index={6}, col sep=comma] {./Results/BasicExample/statistics_fix10e3.csv};
\addlegendentry{SAA, $|S| = 10^3$};

\addplot[color2!75!white, ultra thick, dotted, restrict x to domain=0:1000000] table [x index={2}, y index={6}, col sep=comma] {./Results/BasicExample/statistics_fix10e5.csv};
\addlegendentry{SAA, $|S| = 10^5$};

\addplot[color3, ultra thick, restrict x to domain=0:1000000] table [x index={2}, y index={6}, col sep=comma] {./Results/BasicExample/statistics_050.csv};
\addlegendentry{$\theta = 0.5$};

\addplot[color5, ultra thick, restrict x to domain=0:1000000] table [x index={2}, y index={6}, col sep=comma] {./Results/BasicExample/statistics_100.csv};
\addlegendentry{$\theta = 1.0$};

\addplot[color4, ultra thick, restrict x to domain=0:1000000] table [x index={2}, y index={6}, col sep=comma] {./Results/BasicExample/statistics_150.csv};
\addlegendentry{$\theta = 1.5$};

\legend{}; %
\end{axis}
\end{tikzpicture}

		\end{scaletikzpicturetowidth}
	\end{minipage}%
	\caption{
	Comparison of the stochastic approximation with fixed sample sizes and \Cref{alg:practical} applied to the stochastic program~\cref{eq:basic_example}.
	The top-left and bottom-left plots show the error in the solution vs. the iteration number and the error in the objective function vs. the iteration number, respectively.
	The bottom-right plot shows the error in the objective function vs. the cumulative number of gradient evaluations.
	\label{fig:basic}}
\end{figure}

The leftmost plots in \Cref{fig:basic} illustrate q-linear convergence for each of the adaptive sampling runs, albeit, at different levels of efficiency.
Recalling \Cref{thm:StronglyConvex}, this is the best outcome one could hope for.
For all smaller values of $\theta>0$, the algorithm continues to converge linearly, however, for larger values of $\theta$, the convergence eventually breaks down.
The value of $\theta$ where linear convergence fails depends on the step size $\alpha$, as one would expect from \Cref{thm:StronglyConvex}.
In contrast, the fixed sample size examples with $|S| = 10^1,~10^3$ eventually stop converging. The same would happen in the case of the $|S| = 10^5$ example, given enough iterations. Since we use a fixed step size here, this is the expected behaviour.

The plots on the right in \Cref{fig:basic} provide the sample sizes and resulting gradient evaluations used to obtain the results shown on the left. In the top-right, we observe that the adaptive algorithm increases the sample size roughly exponentially. This behaviour can be interpreted positively from \cite[Section 5]{bottou2018optimization}. Indeed, assuming a uniform bound on the individual gradient samples' variance, an exponentially increasing sample size leads to the variance of the resulting gradient estimate decreasing exponentially. This allows \Cref{alg:practical} to converge linearly and, in this regard, outperform the fixed-sample size algorithm.
The bottom-right plot shows the number of gradient evaluations required for each fixed sample size or value of $\theta$. When using fixed sample and step sizes, the error in the objective function will eventually stop decreasing.
This is avoided when using our adaptive sampling strategy.
Moreover, the number of computed gradient samples can be significantly reduced by adopting an adaptive sample size rule, especially in the early stages of the optimization, when the objective function error is still large.
As a rule of thumb in choosing the adaptive sampling parameter $\theta$, we suggest that one starts with a value around $1.0$ and then track the adaptive algorithm until the first significant growth in the sample size plateaus.
If there has already been a meaningful decrease in the objective value by this point, keep $\theta$ fixed; otherwise, $\theta$ should probably be decreased moderately.
In all cases we have looked at, a reasonable value for $\theta$ can be chosen based on the behavior of the algorithm in its first 10 to 20 iterations.

\subsection{Portfolio optimization}
\label{subsec:operations_research}

With this set of optimization problems, we continue to illustrate the practicality of the adaptive sampling algorithm proposed above.
Specifically, we choose to focus on a class of archetypal operations research problems taken from \cite[Section~6.1]{Royset2013}.
\revised{In this example,} we incorporate the paradigm of risk-averse stochastic optimization; cf. \Cref{sec:risk_averse_problems}.

\subsubsection{Problem description} %
\label{ssub:problem_description}

Let us consider a random cost model with $\vardim=100$ financial instruments whose outputs are each given as $\rv = A + B u$.
In this model, $A$ is a $\vardim$-dimensional vector representing the expected rate of return of a single instrument and $B$ is an $\vardim\times\vardim$-dimensional matrix which correlates the uncertainty in this return.
Each component of $A$ is defined through an independent sample of a uniform distribution over $[0.9,1.2]$ and, likewise, each entry in $B$ is defined by an independent sample of a uniform distribution over $[0,0.1]$.
As with the model parameters $a_l$ and $b_l$ appearing in~\cref{eq:Basic_f}, both $A$ and $B$ only specify parameters in the model.
Therefore, $A$ and $B$ are randomly generated and then held fixed throughout the entire optimization process.
Finally, each component of the $\vardim$-dimensional random vector $u$, which itself acts to introduce uncertainty in the model, is taken to be independent and obey a standard normal distribution.

Given the financial instrument model described above, we now consider the investment of one share of wealth distributed over the $n=100$ independent random financial instruments.
We choose to denote the amount of investment into the $l$-th asset by $\var^l\geq 0$, whereby $\sum_{l=1}^{100}{\var^l} = 1$.
Accordingly, we arrive at the following (stochastic) loss function:
\begin{equation}
	\obj(\var;\rv)=-\sum_{l=1}^{100}{\rv^l \var^l},
\label{eq:stochastic_loss_function}
\end{equation}
where $\var=(\var^1,\dots,\var^{100})$ is our given portfolio allocation strategy.

Let us say that we would like to minimize the loss over all portfolio strategies which have an expected return no smaller than $1.05$.
We therefore define the following admissible set of normalized portfolios:
\begin{equation}
	C = \left\{ \var \in \R^{100} : \var^l \geq 0,\quad \sum_{l=1}^{100}{\var^l} = 1, \quad \sum_{l=1}^{100}{A_l \var^l} \geq 1.05 \quad l=1,\dots,100 \right\}.
\end{equation}

In a risk-neutral paradigm, we seek only to minimize the expected value of~\cref{eq:stochastic_loss_function} over $C$.
The corresponding stochastic program is simply
\begin{equation}
	\min_{\var\in C} ~ \Big\{ \Obj(\var) = \E[\obj(\var;\rv)] \Big\}.
	\label{eq:finance-E}
\end{equation}
With this definition of $F(x)$, the strong convexity assumption made in~\Cref{thm:StronglyConvex} is not satisfied.

It turns out that the expected loss problem above tends not to serve well for most practical investment decisions.
Alternatively, one can minimize the loss with $\CVaR_\beta$ as a risk measure; this is a common choice in financial applications \cite{dowd2007measuring,follmer2011stochastic}.
Accordingly, we focus on the following class of risk-averse stochastic programs:
\begin{equation}
	\min_{\var\in C} ~ \Big\{ \Obj_\beta(\var) = \CVaR_\beta[\obj(\var;\rv)] \Big\}
	,
	\label{eq:finance-CVaR}
\end{equation}
where $\beta\in[0,1)$ is the risk-averseness parameter.
Note that $F(x) = F_0(x)$ in the notation of~\cref{eq:finance-E,eq:finance-CVaR}, and so the risk-neutral program~\cref{eq:finance-E} has not actually been ignored \cite{Rockafellar2000,rockafellar2002conditional}.

\begin{remark}
	As already pointed out in \Cref{sec:risk_averse_problems}, the $\CVaR$ risk measure introduces non-smoothness into the objective functional which commonly breaks the convergence of traditional gradient descent algorithms.
	Therefore we follow \Cref{sub:smoothing} in our experiments and replace the $\CVaR$ in~\cref{eq:finance-CVaR} by the risk measure $\text{CVaR}^\varepsilon_{\beta}$, defined in~\cref{eq:smooth_cvar}, with some small regularization parameter $\varepsilon > 0$.
\end{remark}

\subsubsection{Risk-averse portfolio optimization}
\label{ssub:risk_averse_portfolio_optimization}

In our first set of portfolio optimization experiments, we compare the performance of \Cref{alg:practical} on the stochastic program~\cref{eq:finance-CVaR}, for a variety of risk-averseness parameters $\beta>0$.
Recall~\cref{eq:smooth_cvar} and note that each of these problems may be written as
\begin{equation}
	\min_{(\var,t) \in C\times\R}
	~ \Big\{ F_\beta(x,t) = t + \frac{1}{1-\beta} \: \mathbb{E}\left[(\obj(\var;\rv) - t)_+^\varepsilon\right] \Big\}
	,
\label{eq:portfolio_opt_extended_cvar}
\end{equation}
after regularization.
Because our experiments in \Cref{subsec:basic_example} already indicated a robustness with respect to the algorithm parameter $\theta$, we choose to focus our attention here on its sensitivity to the risk-averseness parameter $\beta$.
In this example, we consider $\beta = 0,~0.5,~0.9,$ and $0.95$.
For all $\beta>0$, we set $\varepsilon = 0.01$.
The value for $\varepsilon$ is chosen as a compromise between a small error w.r.t. the true CVaR and a well-behaved objective function. A bound for this error was introduced in \cite{kouri2016risk}, see also \Cref{rem:bound}.
\revised{The step size is kept constant as before, this time set to $\alpha=0.5$. The value was manually selected such that the simulations give reasonable results, but was not extensively tuned.}

Note that $\beta$ actually changes the optimization problem.
Hence, for each $\beta$, we choose a different sampling rate parameter $\theta$.
For $\beta=0$ and $\beta=0.5$, we set $\theta=2.0$; for $\beta = 0.9$, we take $\theta=1.5$; and for $\beta=0.95$, we specify $\theta=0.125$.
These parameters were chosen to promote comparable growth of the sample size across the different cases.
Generally, as $\beta < 1$ grows, $\theta$ should shrink.
The effect of the sampling rate $\theta$ on the sample size is illustrated further in \Cref{app:sampling_rate}.

\Cref{fig:finance-risk_measures} presents the results of our numerical experiments.
When the risk averseness parameter $\beta$ is increased, we expect that the achieved objective value increases, too.
This feature is clearly observed in~\Cref{fig:finance-risk_measures}.
It is also evident from \Cref{fig:finance-risk_measures} that the initial value of the objective function, $F_\beta(x_0,t_0)$, moves progressively further from its optimal value as $\beta$ grows.
Additionally, for $\beta=0.9$ and $\beta=0.95$, the algorithm fails to improve the objective during the first few iterations.
Both effects are largely due to the fact that the same initial value of the auxiliary variable, $t=t_0$, has been used in each experiment.
In turn, the algorithm takes longer to converge as $\beta$ increases.
It should be noted that the auxiliary variable is optimized using the same step size as used with the spatial variable $x$ without accounting for possibly different scales.

This experiment shows that the performance of \Cref{alg:practical} is very sensitive to the choice of the initial value $t_0$.
In light of \Cref{sub:alternative_algorithm}, one could set $t_0$ to be the solution of~\cref{eq:optimal_t} to obtain a first estimate of the optimal value of $t$ and thus circumvent this issue.
Of course, however, optimizing for $t$ independently of $x$ is essentially what is done at each iteration of \Cref{alg:alternative_cvar}.

\begin{figure}
	\centering
	\begin{minipage}{0.3\textwidth}
		\centering
		\begin{scaletikzpicturetowidth}{\textwidth}
			\begin{tikzpicture}[scale=\tikzscale,font=\large]
\begin{axis}[
xlabel={Iteration},
ylabel={$F_\beta(x_k,t)$},
xmajorgrids,
ymajorgrids,
legend cell align={left},
legend style={fill=white, fill opacity=0.6, draw opacity=1, text opacity=1},
legend pos=north east
]
\addplot[color1, ultra thick] table [x index={0}, y index={3}, col sep=comma, restrict x to domain=0:45] {./Results/FinanceExample/E/statistics.csv};
\addlegendentry{$\beta=0.0$};

\addplot[color2, ultra thick] table [x index={0}, y index={3}, col sep=comma, restrict x to domain=0:60] {./Results/FinanceExample/CVaR-classic/statistics_050.csv};
\addlegendentry{$\beta=0.5$};

\addplot[color3, ultra thick] table [x index={0}, y index={3}, col sep=comma, restrict x to domain=0:60] {./Results/FinanceExample/CVaR-classic/statistics_090.csv};
\addlegendentry{$\beta=0.9$};

\addplot[color8, ultra thick] table [x index={0}, y index={3}, col sep=comma, restrict x to domain=0:100] {./Results/FinanceExample/CVaR-classic/statistics_095.csv};
\addlegendentry{$\beta=0.95$};

\end{axis}
\end{tikzpicture}
		\end{scaletikzpicturetowidth}
	\end{minipage}%
	~
	\begin{minipage}{0.3\textwidth}
		\centering
		\begin{scaletikzpicturetowidth}{\textwidth}
			\begin{tikzpicture}[scale=\tikzscale,font=\large]
	\begin{axis}[
		xlabel={Gradient evaluations},
		ylabel={$F_\beta(x_k,t)$},
		xmajorgrids,
		ymajorgrids,
		legend cell align={left},
		xmode = log,
		legend style={fill=white, fill opacity=0.6, draw opacity=1, text opacity=1},
		legend pos=north east
		]
		\addplot[color1, ultra thick] table [x index={2}, y index={3}, col sep=comma] {./Results/FinanceExample/E/statistics.csv};
		\addlegendentry{$\beta=0.0$};
		
		\addplot[color2, ultra thick] table [x index={2}, y index={3}, col sep=comma] {./Results/FinanceExample/CVaR-classic/statistics_050.csv};
		\addlegendentry{$\beta=0.5$};
		
		\addplot[color3, ultra thick] table [x index={2}, y index={3}, col sep=comma] {./Results/FinanceExample/CVaR-classic/statistics_090.csv};
		\addlegendentry{$\beta=0.9$};
		
		\addplot[color8, ultra thick] table [x index={2}, y index={3}, col sep=comma] {./Results/FinanceExample/CVaR-classic/statistics_095.csv};
		\addlegendentry{$\beta=0.95$};
		
		\legend{}; %
	\end{axis}
\end{tikzpicture}
		\end{scaletikzpicturetowidth}
	\end{minipage}%
	~
	\begin{minipage}{0.3\textwidth}
		\centering
		\begin{scaletikzpicturetowidth}{\textwidth}
			\begin{tikzpicture}[scale=\tikzscale,font=\large]
\begin{axis}[
xlabel={Iteration},
ylabel={Sample size},
xmajorgrids,
ymajorgrids,
ymode = log,
legend style={fill=white, fill opacity=0.6, draw opacity=1, text opacity=1},
legend pos=north west
]
\addplot[color1, ultra thick] table [x index={0}, y index={1}, col sep=comma, restrict x to domain=0:45] {./Results/FinanceExample/E/statistics.csv};
\addlegendentry{$\beta=0.0$};

\addplot[color2, ultra thick] table [x index={0}, y index={1}, col sep=comma] {./Results/FinanceExample/CVaR-classic/statistics_050.csv};
\addlegendentry{$\beta=0.5$};

\addplot[color3, ultra thick] table [x index={0}, y index={1}, col sep=comma] {./Results/FinanceExample/CVaR-classic/statistics_090.csv};
\addlegendentry{$\beta=0.9$};

\addplot[color8, ultra thick] table [x index={0}, y index={1}, col sep=comma] {./Results/FinanceExample/CVaR-classic/statistics_095.csv};
\addlegendentry{$\beta=0.95$};

\legend{}; %
\end{axis}
\end{tikzpicture}
		\end{scaletikzpicturetowidth}
	\end{minipage}%

	\caption{Numerical results for the portfolio optimization problem~\cref{eq:finance-CVaR} for the risk-averseness parameters $\beta=0.0$ (which also corresponds to the risk-neutral problem~\cref{eq:finance-E}), $0.5,~0.9,$ and $0.95$.
	On the left, we see the value of the objective function converge with respect to the iteration number.
	In the middle, we see the value of the objective function converge with respect to the cumulative number of gradient evaluations.
	On the right, we see the sample size grow with respect to the iteration number.
	\label{fig:finance-risk_measures}
	}
\end{figure}

\subsubsection[Risk-averse portfolio optimization with Algorithm 3]{Risk-averse portfolio optimization with \Cref{alg:alternative_cvar}}
\label{sssub:risk_avere_portfolio_optimization_Algo4}
Here, we briefly compare \Cref{alg:alternative_cvar} to \Cref{alg:practical}.
Since the setting $\beta = 0$ simply amounts to problem~\cref{eq:finance-E}, our comparison only involves $\beta=0.5,~0.9$, and $0.95$.
It may be seem natural to also choose the same values of $\theta$ used in~\Cref{ssub:risk_averse_portfolio_optimization}, however, we found that the auxiliary variable $t$ appearing in~\cref{eq:portfolio_opt_extended_cvar} has a strong effect on variance of the objective function.
Because the gradient with respect to $t$ does not appear in \Cref{alg:alternative_cvar}, we were able to use larger values of $\theta$ than in~\Cref{ssub:risk_averse_portfolio_optimization} and this tended to result in better sample size efficiency.
To be specific, for $\beta=0.5$, we set $\theta=4.0$; for $\beta=0.9$, we set $\theta = 4.5$; and for $\beta=0.95$, we set $\theta=4.5$.

\begin{figure}
	\centering
	\begin{minipage}{0.3\textwidth}
		\centering
		\begin{scaletikzpicturetowidth}{\textwidth}
			\begin{tikzpicture}[scale=\tikzscale,font=\large]
\begin{axis}[
xlabel={Iteration},
ylabel={$F_\beta(x_k,t)$},
xmajorgrids,
ymajorgrids,
legend columns=2,
legend cell align={left},
legend style={fill=white, fill opacity=0.9, draw opacity=1, text opacity=1, font=\normalsize},
legend pos=north east
]
\setlength{\fboxrule}{0pt}
\addlegendimage{legend image with text=\fbox{Alg. \ref{alg:practical}}}
\addlegendentry{}
\addlegendimage{legend image with text=\fbox{Alg. \ref{alg:alternative_cvar}}}
\addlegendentry{}

\addplot[color2!70!white, ultra thick, dashed] table [x index={0}, y index={3}, col sep=comma ] {./Results/FinanceExample/CVaR-classic/statistics_050.csv};
\addlegendentry{};
\addplot[color2, ultra thick] table [x index={0}, y index={3}, col sep=comma,] {./Results/FinanceExample/CVaR-new/statistics_050.csv};
\addlegendentry{$\beta = 0.5$};

\addplot[color3!70!white, ultra thick, dashed] table [x index={0}, y index={3}, col sep=comma] {./Results/FinanceExample/CVaR-classic/statistics_090.csv};
\addlegendentry{};
\addplot[color3, ultra thick] table [x index={0}, y index={3}, col sep=comma] {./Results/FinanceExample/CVaR-new/statistics_090.csv};
\addlegendentry{$\beta = 0.9$};

\addplot[color8!70!white, ultra thick, dashed] table [x index={0}, y index={3}, col sep=comma] {./Results/FinanceExample/CVaR-classic/statistics_095.csv};
\addlegendentry{};
\addplot[color8, ultra thick] table [x index={0}, y index={3}, col sep=comma,, restrict x to domain=0:45] {./Results/FinanceExample/CVaR-new/statistics_095.csv};
\addlegendentry{$\beta = 0.95$};

\end{axis}
\end{tikzpicture}
		\end{scaletikzpicturetowidth}
	\end{minipage}%
	~
	\begin{minipage}{0.3\textwidth}
		\centering
		\begin{scaletikzpicturetowidth}{\textwidth}
			\begin{tikzpicture}[scale=\tikzscale,font=\large]
\begin{axis}[
xlabel={Gradient evaluations},
ylabel={$F_\beta(x_k,t)$},
xmajorgrids,
ymajorgrids,
legend columns=2,
legend cell align={left},
xmode = log,
legend style={fill=white, fill opacity=0.9, draw opacity=1, text opacity=1, font=\small},
legend pos=north east
]
\setlength{\fboxrule}{0pt}
\addlegendimage{legend image with text=\fbox{Alg. \ref{alg:practical}}}
\addlegendentry{}
\addlegendimage{legend image with text=\fbox{Alg. \ref{alg:alternative_cvar}}}
\addlegendentry{}

\addplot[color2!70!white, ultra thick, dashed] table [x index={2}, y index={3}, col sep=comma ] {./Results/FinanceExample/CVaR-classic/statistics_050.csv};
\addlegendentry{};
\addplot[color2, ultra thick] table [x index={2}, y index={3}, col sep=comma,] {./Results/FinanceExample/CVaR-new/statistics_050.csv};
\addlegendentry{$\beta = 0.5$};

\addplot[color3!70!white, ultra thick, dashed] table [x index={2}, y index={3}, col sep=comma] {./Results/FinanceExample/CVaR-classic/statistics_090.csv};
\addlegendentry{};
\addplot[color3, ultra thick] table [x index={2}, y index={3}, col sep=comma] {./Results/FinanceExample/CVaR-new/statistics_090.csv};
\addlegendentry{$\beta = 0.9$};

\addplot[color8!70!white, ultra thick, dashed] table [x index={2}, y index={3}, col sep=comma] {./Results/FinanceExample/CVaR-classic/statistics_095.csv};
\addlegendentry{};
\addplot[color8, ultra thick] table [x index={2}, y index={3}, col sep=comma] {./Results/FinanceExample/CVaR-new/statistics_095.csv};
\addlegendentry{$\beta = 0.95$};

\legend{}; %
\end{axis}
\end{tikzpicture}
		\end{scaletikzpicturetowidth}
	\end{minipage}%
	~
	\begin{minipage}{0.3\textwidth}
		\centering
		\begin{scaletikzpicturetowidth}{\textwidth}
			\begin{tikzpicture}[scale=\tikzscale,font=\large]
\begin{axis}[
xlabel={Iteration},
ylabel={Sample size},
xmajorgrids,
ymajorgrids,
ymode = log,
legend style={fill=white, fill opacity=0.6, draw opacity=1, text opacity=1},
legend pos=north west
]

\addplot[color2!70!white, ultra thick, dashed] table [x index={0}, y index={1}, col sep=comma] {./Results/FinanceExample/CVaR-classic/statistics_050.csv};
\addlegendentry{Algorithm \ref{al:smooth_cvar}};
\addplot[color2, ultra thick] table [x index={0}, y index={1}, col sep=comma] 
{./Results/FinanceExample/CVaR-new/statistics_050.csv};
\addlegendentry{Algorithm \ref{al:smooth_cvar-classic}};

\addplot[color3!70!white, ultra thick, dashed] table [x index={0}, y index={1}, col sep=comma] {./Results/FinanceExample/CVaR-classic/statistics_090.csv};
\addlegendentry{Algorithm \ref{al:smooth_cvar}};
\addplot[color3, ultra thick] table [x index={0}, y index={1}, col sep=comma] 
{./Results/FinanceExample/CVaR-new/statistics_090.csv};
\addlegendentry{Algorithm \ref{al:smooth_cvar-classic}};

\addplot[color8!70!white, ultra thick, dashed] table [x index={0}, y index={1}, col sep=comma] {./Results/FinanceExample/CVaR-classic/statistics_095.csv};
\addlegendentry{Algorithm \ref{al:smooth_cvar}};
\addplot[color8, ultra thick] table [x index={0}, y index={1}, col sep=comma] 
{./Results/FinanceExample/CVaR-new/statistics_095.csv};
\addlegendentry{Algorithm \ref{al:smooth_cvar-classic}};

\legend{}; %
\end{axis}
\end{tikzpicture}
		\end{scaletikzpicturetowidth}
	\end{minipage}%
	\caption{Numerical results for problem~\cref{eq:finance-CVaR}, with $\beta=0.5,~0.9$, and $0.95$, comparing \Cref{alg:practical,alg:alternative_cvar}. \label{fig:finance-alternative}}
\end{figure}

The results of our comparison are presented in~\Cref{fig:finance-alternative}.
Evidently, both algorithms converge to the same optimal objective value.
Although \Cref{alg:alternative_cvar} involves solving a one-dimensional optimization problem at each iteration, it also appears to generate fewer samples which could make it more efficient overall in some applications.
\section{Adaptive sampling with non-convex constraints} %
\label{sec:nonconvex-applications}

Up to this point, the convexity of the constraint set $C$ has been critical.
In general, it is even required to uniquely define the orthogonal projection~\cref{eq:ProjectionOperator}.
Nevertheless, many optimization problems involve non-convex constraints, and we seek to show that some of the ideas introduced above can still be used in that setting.
Our treatment is not intended to be comprehensive; we give one practical example and leave its generalizations for future study.

\subsection{Treatment of non-convex constraints} %
\label{sub:non_convex_problems}

Assume that we are required to optimize over the level set of a smooth function $G: \R^n \to \R$, \revised{and so} we rewrite~\cref{eq:P-C} as
\begin{equation}
	\min_{\var\in \R^n}~
	\Big\{ F(\var) = \bbE[\obj(\var;\rv)] 
	\quad 
	\text{subject to } G(\var)=0
	\Big\}
	.
\end{equation}
\revised{Let us also assume that the gradient of $G$ never vanishes on the feasible set.}
\revised{In this case, the linear indepence constraint qualification (LICQ) is trivially satisfied because the set of constraint gradients is always one-dimensional and the above} problem can be solved with sequential quadratic programming (SQP) principles \cite{nocedal2006numerical,hinze2008optimization,Ulbrich2012}.

\revised{For simplicity, let $B_k$ be a symmetric positive definite (SPD) matrix}\footnote{\revised{It is well-known that we may relax this requirement to being that $B_k$ is only SPD on the tangent space of the constraint set; cf. \cite{powell1983variable,nocedal2006numerical}.}} and define $\lVert d \rVert_{B_k} = \sqrt{\langle d, B_k d \rangle}$.
Typically, we seek the optimum
\begin{equation}
	d_{S_k}
	=
	\argmin_{d\in \R^\vardim}~
	\langle \nabla \Obj_\subsampling(\varit), d \rangle
	+
	\frac{1}{2}\lVert d \rVert_{B_k}^2
	\quad
	\text{subject to } \langle \nabla G(\varit), d\rangle + G(\varit) = 0
	,
\label{eq:SQP}
\end{equation}
and update the solution $x_k \mapsto x_k + \alpha d_{S_k}$.
In its most basic form \cite{powell1983variable}, we may assume that each $B_k = I$ and so $\lVert d \rVert_{B_k} = \lVert d \rVert$.
In this setting, a straightforward computation shows that $d_{S_k} = - R_{S_k}(x_k)$ when the affine \revisedBW{subspace} $$C_k = \{y \in \R^n \colon \langle \nabla G(\varit), y - x_k\rangle + G(\varit) = 0 \}$$ is substituted for $C$ in definition~\cref{eq:GradientMap_sampled}.
This observation establishes a well-known connection between projected gradient algorithms and SQP \cite{powell1983variable}.
As such, it also provides a connection between the preceding analysis and a treatment of non-convex constraints where a linearized constraint space is updated at each iteration $k$.
Throughout the rest of this section, when we refer to $R_{S_k}(x_k)$ or $R(x_k)$, we assume that $C = C_k$.

As stated in~\Cref{rem:Affine}, \Cref{con:Bias} is trivially satisfied when $C$ is \revisedBW{an affine subspace}.
Of course, this happens to be the case in~\cref{eq:SQP} because every $C_k$ is \revisedBW{an affine subspace}.
It is interesting to note that an affine \revisedBW{subspace} constraint makes it possible to propose alternatives to \Cref{con:NEWcondition} that are less restrictive on the size of the sample set $S_k$.
One possibility is to propose a threshold on the expected value of $R_{S_k}(x_k)$ lying within a ball around $R(x_k)$.
This may be written as follows:
\smallskip
\begin{condition}
\label{con:NormTest2}
	Control of the error in the reduced gradient by the norm of the reduced gradient:
	\begin{equation}
	\label{eq:ExactNormTest}
		\bbE_k\left[\lVert \gradCsub(\varit) - \gradC(\varit) \rVert^2\right]
		\leq
		\theta^2 \lVert \gradC(\varit) \rVert^2
		,
	\end{equation}
	for some fixed $\theta > 0$.
\end{condition}
\smallskip
We note that $\bbE_k[R_{S_k}(x_k)] = R(x_k)$ by the affine nature of $P:\R^n \to C_k$.
Therefore, in this specific setting, we have the identity
\begin{equation}
	\bbE_k\left[\lVert R_{S_k}(x_k) - R(x_k) \rVert^2\right]
	=
	\bbE_k\left[\lVert R_{S_k}(x_k)\rVert^2\right] - \bbE_k\left[\lVert R(x_k) \rVert^2\right]
	,
\end{equation}
and so~\Cref{con:NormTest2} is actually equivalent to~\Cref{con:Norm} with $\nu = \theta$.
Regardless, to avoid confusing with the general setting where~\Cref{con:Norm,con:NormTest2} are not equivalent, we choose to denote~\Cref{con:Norm,con:NormTest2} differently.

\subsection{\revised{Practical algorithm}}

\revised{In order to derive a\revisedBW{n SQP-type projected gradient algorithm of practical relevance} with adaptive sampling relying on \Cref{sub:non_convex_problems},} let us define
\begin{equation}
\label{eq:SQP_sample}
	d(x_k;\xi)
	=
	\argmin_{d\in \R^\vardim}~
	\langle \nabla f(x_k;\xi), d \rangle
	+
	\frac{1}{2}\lVert d \rVert^2
	\quad
	\text{subject to } \langle \nabla G(\varit), d\rangle + G(\varit) = 0
\end{equation}
and, accordingly, $R(x_k;\xi) = -d(x_k;\xi)$.
We then propose the following test which may be used to check \Cref{con:NormTest2}.
\smallskip
\begin{test}[Approximation of \Cref{con:NormTest2}]
\label{con:NormTest2_approx}
	Approximate control of the error in the reduced gradient by the norm of the reduced gradient:
	\begin{equation}
	\label{eq:NormTest2_inexact}
		\frac{1}{|S_k|-1}\frac{\sum_{\xi\in S_k}\|R(x_k;\xi) - R_{S_k}(x_k)\|^2}{|S_k|}
		\leq
		\theta^2 \|R_{S_k}(x_k)\|^2
		,
	\end{equation}
	for some fixed $\theta > 0$.
\end{test}
\smallskip

\begin{remark}
\Cref{con:NormTest2_approx} may appear undesirable because it involves computing individual reduced gradients $R(x_k;\xi)$ and, thus, repeated applications of the projection operator $P:\R^n \to C_k$.
This, however, it not a deep concern since a projection onto an affine subspace is usually very cheap to evaluate and can often be performed via sparse matrix operations \cite{trefethen1997numerical}.
\end{remark}

\begin{remark}
One could also propose other alternatives to \Cref{con:NormTest2}.
For instance, one could follow Bollapragada et al. \cite{Bollapragada2018a} and derive conditions which lead to a probabilistic threshold on $R_{S_k}(x_k)$ pointing in the same direction as $R(x_k)$.
For sake of space, we do not include any algorithms based on this approach.
The interested reader is referred to \cite{Bollapragada2018a,Urbainczyk2020} for further details on how such algorithms could be constructed.

\end{remark}

\paragraph{\revised{Feasibility}}
\revised{
Algorithms for constrained optimization problems need to verify both optimality and feasibility conditions.
Since $G$ is generally not an affine function, the proposed update steps $d_{S_k}$ \revised{alone} do not necessarily result in feasible iterates $x_k$.
Moreover, we assume that we do not have the ability to perform an exact projection \revisedBW{onto} the feasible set.
}

One way to proceed is to add a correction term $g_\its$ to each step proposal with the aim of approximating the constraint rather than satisfying it exactly \cite{rosen1961gradient}. 
This is the approach chosen in the KratosMultiphysics Shape Optimization Application \cite{antonau2022latest} \revised{and is reproduced here as part of \Cref{alg:adaptive_sampling_SQP}}.
The algorithm introduces a scaling parameter $\psi_\its$ for each iteration $\its$ which determines the magnitude of the correction term $g_\its$. 
Then, in an \textbf{if} statement, we determine whether the value of $G$ has changed its sign \revised{following the last update step}. 
If this is true, we are close to the \revised{level set $C$}. 
As a consequence, the parameter $\psi_\its$, and thus the magnitude of the correction term, are reduced. 
The \revised{subsequent} \textbf{else if} statement checks the opposite case, that is, whether we have drifted further from the constraint manifold during the two last \revised{step updates}. 
In that case, the correction scaling is increased to counteract this drifting. 
As a last step, the correction term is calculated as the scaled gradient of the constraint function $\nabla G(x_k)$. 
This correction is then deducted from the proposed update step $d_{S_\its}$ to steer the iterates closer to the feasible set.

\smallskip
\revised{\Cref{alg:adaptive_sampling_SQP} below adopts the SQP-inspired, projected gradient optimization step with an adaptive sample size relying on \Cref{con:NormTest2_approx} and the correction scheme described above in a practical manner.}

\begin{algorithm2e}[H]
\DontPrintSemicolon
	\caption{\label{alg:adaptive_sampling_SQP}Adaptive sampling algorithm for stochastic programs with functional equality constraints}
	\SetKwInOut{Input}{input}
	\Input{Feasible $\var_0 \in C$, step size $\alpha>0$, initial sample set $S_0$, constant $\theta > 0$, scaling parameter $\psi_0 > 0$.}
	Set $\its \leftarrow 0$\revised{, $g_0 \leftarrow 0^n\in\mathbb{R}^n$ and compute $d_{S_0}$}.\;
	\Repeat{\revised{a given} convergence test is satisfied}
	{
		\If{\Cref{con:NormTest2_approx} is satisfied}
		{
			Update $\varitnext = \varit + \alpha d_{S_k} \revised{ - \alpha \|d_{S_k}\| g_\its}$.\;
			Set $\its \leftarrow \its+1$.\;
			Construct $S_{k}$ satisfying $|S_{k}| = |S_{k-1}|$.\;
			\revised{
			\If{$G(x_\its)\:G(x_{\its-1}) < 0$}
			{Set $\psi_{\its+1} \leftarrow \frac{1}{2}\psi_\its \: .$}
			\ElseIf{$G(x_\its)\:G(x_{\its-1}) > 0$ \bf{and} $\|G(x_\its)\| > \|G(x_{\its-1})\|$}
			{Set $\psi_{\its+1} \leftarrow \min\:(2\psi_\its, 1) \: .$}
			Compute correction term $g_\its = \mathrm{sign}(G(x_\its)) \: \psi_{\its+1} \frac{\nabla G(x_\its)}{\|\nabla G(x_\its)\|} \: .$\;
			}
		}
		\Else
		{Set $|S_k| \leftarrow \lceil \rho^\prime\,|S_k|\rceil$, where
		$
			\rho^\prime
			=
			\dfrac{\sum_{\xi\in S_k}\|R(x_k;\xi) - R_{S_k}(x_k)\|^2}{\theta^2 \: (|S_k|-1) \: |S_k| \: \|R_{S_k}(x_k)\|^2} \:
			.
		$\;
		Obtain additional i.i.d. samples for $S_{k}$.\;
		}
		\revised{(Re-)compute $d_{S_\its}$.\;}
	}
\end{algorithm2e}

Note that \revised{in \Cref{alg:adaptive_sampling_SQP},} the sample set $S_k$ is updated \emph{before} progressing to the next iteration.
This is a desirable choice when collecting samples is more expensive than solving~\cref{eq:SQP} or~\cref{eq:SQP_sample}, which happens to be the case in the \revised{following} engineering application.%

\subsection{Shape optimization of shell structures}
\label{sub:shell}

Finally, we turn our attention towards a problem in engineering shape design.
\revised{Specifically, we consider} the design of a thin steel shell structure with physical model uncertainties.

It is well-accepted that shape optimization problems are difficult to characterize as well as solve and often involve significant engineering oversight \cite{Bletzinger2017}.
The intention in such problems is usually not to seek a globally optimal design, but instead to begin with an initial ``good'' design and find a nearby local optimum, which improves on a specified quantity of interest.
\revised{This is the setting in which we test the performance of \Cref{alg:adaptive_sampling_SQP}.}

\smallskip
\paragraph{\revised{Problem description}}
Our chosen example centers on the question of how to find the shape of a steel shell which minimizes some measurement of the internal strains resulting from a specified distribution of external loads.
For the initial shape, we choose a half-cylinder on its side, as depicted in \Cref{fig:shell_initial}.
We assume that an uncertain load $\bff$ will be applied to the shell structure from above and that the final manufactured thickness $t$ of the shell is also uncertain.
To simplify our implementation, we assume that every cross-section of the applied load follows a simple bell shape profile along the major axis of the shell and that the uncertainty in the load lies only in the position where it achieves its maximum.
More specifically, we model the applied load (measured in Newtons) by the vector field
\begin{align}
	\bff(x,y,z)
	=
	-10^5\exp\big( -( x - 6 + 4 a / 3 ) ^ 2 \big)\bfe_z
	,
	& \quad \text{with } a \sim \sfN(0,1),
	\quad
	\bfe_z = (0,0,1),
\end{align}
and all distances measured in units of meters (m); cf. \Cref{fig:shell_initial}.
Furthermore, we model the uncertain shell thickness by the uniformly distributed random variable
\begin{equation}
	t \sim \mathsf{Unif}(\SI{4.05}{\cm},\SI{5.05}{\cm})
	.
\end{equation}
It is of course possible to consider other uncertain model parameters, in addition to the thickness.
In this example, however, we choose to fix the mass density (\SI{7.85e3}{\kilogram\per\cubic\meter}), Young's modulus (\SI{2.069e11}{\pascal}), and Poisson's ratio (\SI{2.9e-1}{}) of the steel shell structure, judging them to be far less sensitive sources of uncertainty.

\begin{figure}
	\centering
	\begin{minipage}{0.55\textwidth}
		\centering 
		\includegraphics[clip=true, trim= 5cm 0cm 3cm 0cm, width=\textwidth]{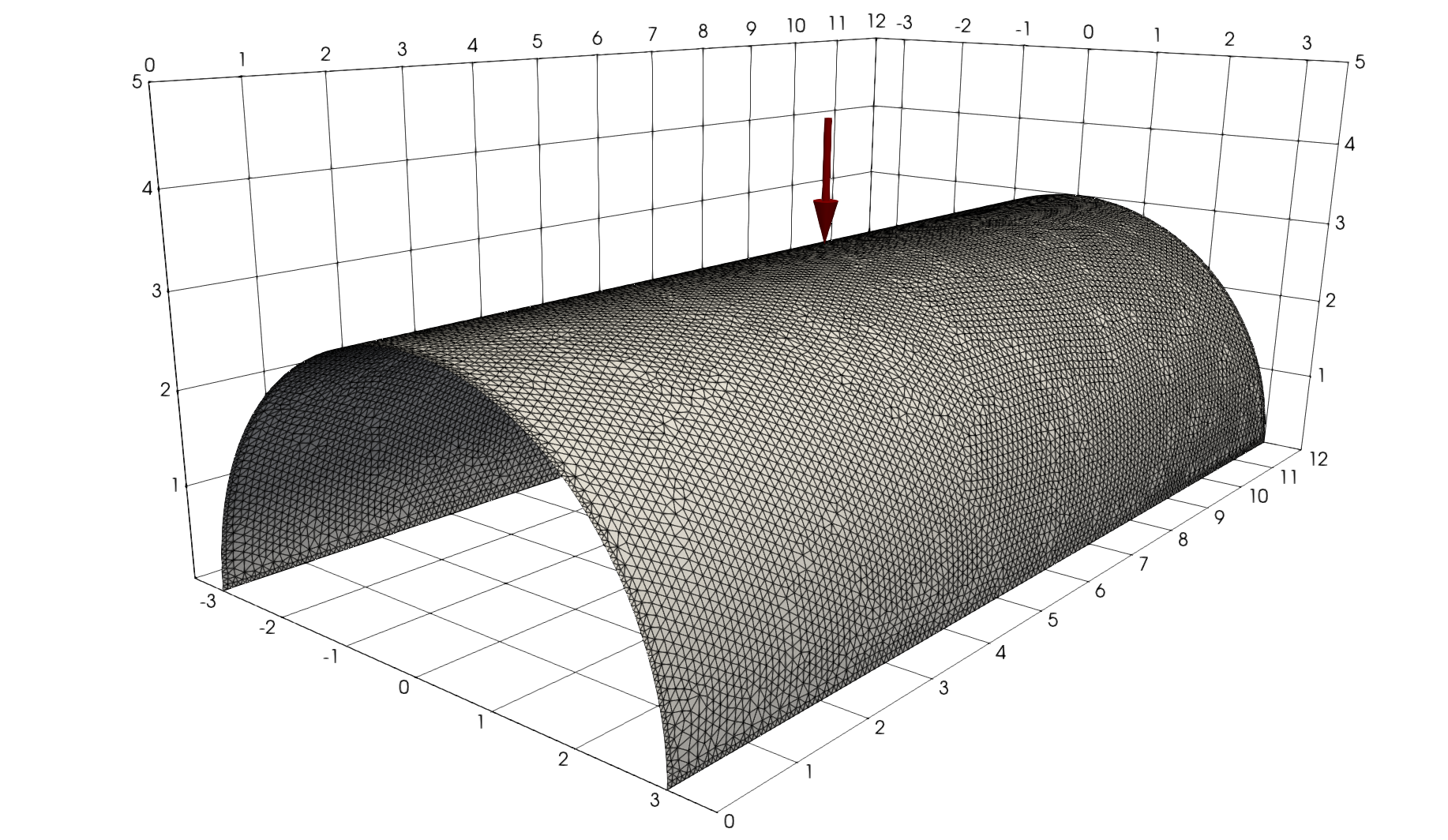}
	\end{minipage}%
	\quad
	\begin{minipage}{0.38\textwidth}
		\centering
		\begin{scaletikzpicturetowidth}{\textwidth}
			\begin{tikzpicture}[scale=\tikzscale,font=\large]
\definecolor{arrow_red}{rgb}{0.43, 0, 0}
\begin{axis}[
xlabel={$x$-axis coordinate (m)},
ylabel={Magnitude of load (N)},
]
\addplot[arrow_red, ultra thick, samples=100, domain=0:12] {10^5*exp(-((x-6)/1)^2)};
\legend{}; %
\end{axis}
\end{tikzpicture}
		\end{scaletikzpicturetowidth}
	\end{minipage}%
	\caption{\label{fig:shell_initial}Initial design of shell structure and applied force, $\bff$, when $a=0$.}
\end{figure}

Our goal here is to find a geometry parameterization $x$ which optimizes the shell's internal energy $\Pi(x) = \Pi(x;\bff,t)$, subject to the stochastic load $\bff$ and thickness $t$, given above.
In order to arrive at a realistic and practical optimum, we only look at a set of similarly expensive geometries \textemdash{} namely, those having (i) equal surface area \textemdash{} and physically reasonable geometries, wherein (ii) the supporting sides of the shell structure remain on the ground and (iii) the open open ends of the structure stay perpendicular to the ground.
These three sets of constraints lead to an abstract design space $C$ and an associated stochastic optimization problem, which may be compactly written as
\begin{equation}
	\min_{\var\in C}~ \Big\{ F(x) = \mcR[\Pi(x)] \Big\}
	,
	\label{eq:shell-opt-problem}
\end{equation}
where $\mcR$ is a given risk measure.
We only consider $\mcR = \CVaR^\varepsilon_{\beta}$, where $\beta \in [0,1)$ and $\varepsilon \geq 0$.

\smallskip
\paragraph{\revised{Experiment setup}}
It is common practice to represent the design geometry by the position of the nodes in its finite element representation \cite{Bletzinger2014}.
These nodes, in turn, serve as control variables $\var \in \R^n$, which may be updated along their physical normals at each step in the optimization algorithm \cite{Bletzinger2017}.
In this example, we follow the semi-analytical adjoint-based procedure outlined in \cite[Subsection~5.5.4]{Bletzinger2017} and implemented in the KratosMultiphysics Structural Mechanics and Shape Optimization Application \cite{Kratos2010}.
The geometry update rule we employ at the end of each optimization step $k$ uses sophisticated filtering techniques and mesh movement algorithms outlined in \cite{Bletzinger2014}.

Using node positions as our design variable allows us to implement constraints (ii) and (iii) quite simply.
By excluding the appropriate coordinates of the side and end nodes from design variable updates, we effectively fix them in the desired planes.
Constraint (i) is encoded via the computation of the geometry's surface area $A(x)$ by requiring that $A(x) = A(x_0)$, with $x_0$ denoting the initial design.
\revised{The starting design is thus feasible by definition.}
Gradients of $A$, which are required for computing the update step $d_{S_k}$ in \cref{eq:SQP}, were evaluated using an internal KratosMultiphysics \cite{Kratos2010} routine.
In practice, we add a correction term to each update step to account for the non-linearity of this constraint,
\revised{as described in \Cref{alg:adaptive_sampling_SQP}.}

Following \cite{Bletzinger2017}, for each independent pair of load and thickness realizations, $\bff_i$ and $t_i$, the calculation of the corresponding strain energy realization, $\Pi_i(x_k) = \Pi(x_k;\bff_i,t_i)$, together with its gradient, $\nabla \Pi_i(x_k)$, involves the discrete solution of two partial differential equations (PDEs).
In this study, we choose to represent the shell using the classical three-parameter Kirchhoff-Love PDE model and form a discretization of it with lowest-order $C^0$-continuous finite elements; cf. \cite{bischoff2018models}.
In particular, we use three-node ANDES elements \cite{felippa2003study} on the $(n=13783)$-node simplicial mesh depicted in \Cref{fig:shell_initial}.

We perform five numerical shape design optimization experiments with the discretization just described.
The first experiment uses $\beta = \varepsilon = 0$. Since for $\beta = 0$ the CVaR is identical with the expectation, we require no smoothing in this case.
The second, third, fourth, and fifth experiments use $\beta=0.5,~0.75,~0.8$, and $0.9$ respectively, each with $\varepsilon = 0.1$.
In the risk-neutral case (i.e., $\beta = 0$, $\mcR = \bbE$), we begin with an initial sample size of $|S_0| = 5$.
In each of the risk-averse cases (i.e., $\beta > 0$, $\mcR = \CVaR_\beta^\varepsilon$), we begin with $|S_0| = 10$ in order to be able to produce meaningful estimates of $t^\ast$ from the start. In the expectation case, the initial sample size is chosen smaller to exhibit a more notable increase, as the sample size barely surpasses 10 towards the last iterations.
In every experiment, we use the step size $\alpha = 0.1$ and the sampling rate parameter $\theta = 0.8$.

Each of the optimization problems is solved using the SQP approach introduced in \Cref{sub:non_convex_problems}.
For the risk-neutral problem, $\beta = 0$, we use \Cref{alg:adaptive_sampling_SQP}.
However, for the CVaR problems, $\beta > 0$, we modify the algorithm with the nested quantile estimation strategy described in \Cref{sub:alternative_algorithm}.
For further details, see \cite[Chapter~4.5]{Urbainczyk2020}.
In each experiment, the stochastic optimization algorithm is stopped after 50 iterations. 
Plots of the optimization logs are given in \Cref{fig:shell} and the final geometries are shown in \Cref{fig:shell-final}.
For visual comparison, we also present the final geometry one would find by optimizing the shell if it had exactly the expected thickness $t = \SI{5.00}{\cm}$, and exactly the expected stress $\bff(x,y,z) = -10^5\exp\big( -( x - 6 ) ^ 2 \big)\bfe_z$ was being applied (i.e., $a = \SI{0.0}{\m}$).
The interested reader may also consult \cite{Urbainczyk2020} for additional shape optimization experiments.

\begin{figure}
	\centering
	\begin{minipage}{0.3\textwidth}
		\centering
		\begin{scaletikzpicturetowidth}{\textwidth}
			\begin{tikzpicture}[scale=\tikzscale,font=\large]
\begin{axis}[
xlabel={Iteration},
ylabel={$\CVaR_{\beta}[\Pi(x_k)]$},
xmajorgrids,
ymajorgrids,
legend cell align={left},
ymode = log,
legend style={fill=white, fill opacity=0.6, draw opacity=1, text opacity=1},
legend pos=north east
]
\addplot[color1, ultra thick] table [x index={6}, y index={3}, col sep=comma, restrict x to domain=0:50] {./Results/ShellExample/expectation_norm.csv};
\addlegendentry{$\beta = 0.0$};
\addplot[color2, ultra thick] table [x index={7}, y index={4}, col sep=comma, restrict x to domain=0:50] {./Results/ShellExample/smooth_cvar_0.5_norm.csv};
\addlegendentry{$\beta = 0.5$};
\addplot[color3, ultra thick] table [x index={7}, y index={4}, col sep=comma, restrict x to domain=0:50] {./Results/ShellExample/smooth_cvar_0.75_norm_alt.csv};
\addlegendentry{$\beta = 0.75$};
\addplot[color4, ultra thick] table [x index={7}, y index={4}, col sep=comma, restrict x to domain=0:50] {./Results/ShellExample/smooth_cvar_0.8_norm.csv};
\addlegendentry{$\beta = 0.8$};
\addplot[color5, ultra thick] table [x index={7}, y index={4}, col sep=comma, restrict x to domain=0:50] {./Results/ShellExample/smooth_cvar_0.9_norm.csv};
\addlegendentry{$\beta = 0.9$};
\end{axis}
\end{tikzpicture}
		\end{scaletikzpicturetowidth}
	\end{minipage}%
	~
	\begin{minipage}{0.3\textwidth}
		\centering
		\begin{scaletikzpicturetowidth}{\textwidth}
			\begin{tikzpicture}[scale=\tikzscale,font=\large]
\begin{axis}[
xlabel={Gradient evaluations},
ylabel={$\CVaR_{\beta}[\Pi(x_k)]$},
xmajorgrids,
ymajorgrids,
legend cell align={left},
xmode = log,
ymode = log,
legend style={fill=white, fill opacity=0.6, draw opacity=1, text opacity=1},
legend pos=north east
]
\addplot[color1, ultra thick, select coords between index={0}{50}] table [x index={5}, y index={3}, col sep=comma] {./Results/ShellExample/expectation_norm.csv};
\addlegendentry{$\E$};
\addplot[color2, ultra thick, select coords between index={0}{50}] table [x index={6}, y index={4}, col sep=comma] {./Results/ShellExample/smooth_cvar_0.5_norm.csv};
\addlegendentry{$\CVaR_{0.5}$};
\addplot[color3, ultra thick, select coords between index={0}{50}] table [x index={6}, y index={4}, col sep=comma] {./Results/ShellExample/smooth_cvar_0.75_norm_alt.csv};
\addlegendentry{$\CVaR_{0.75}$};
\addplot[color4, ultra thick, select coords between index={0}{50}] table [x index={6}, y index={4}, col sep=comma] {./Results/ShellExample/smooth_cvar_0.8_norm.csv};
\addlegendentry{$\CVaR_{0.8}$};
\addplot[color5, ultra thick, select coords between index={0}{50}] table [x index={6}, y index={4}, col sep=comma] {./Results/ShellExample/smooth_cvar_0.9_norm.csv};
\addlegendentry{$\CVaR_{0.9}$};
\legend{}; %
\end{axis}
\end{tikzpicture}
		\end{scaletikzpicturetowidth}
	\end{minipage}%
	~
	\begin{minipage}{0.3\textwidth}
		\centering
		\begin{scaletikzpicturetowidth}{\textwidth}
			\begin{tikzpicture}[scale=\tikzscale,font=\large]
\begin{axis}[
xlabel={Iteration},
ylabel={Sample size},
xmajorgrids,
ymajorgrids,
legend cell align={left},
legend style={fill=white, fill opacity=0.6, draw opacity=1, text opacity=1},
legend style={at={(axis cs:50,15)},anchor=east},
ytick={0,10,20,30,40,50},
ymin = -5,
]
\addplot[color1, ultra thick] table [x index={6}, y index={2}, col sep=comma, restrict x to domain=0:50] {./Results/ShellExample/expectation_norm.csv};
\addlegendentry{$\E$};
\addplot[color2, ultra thick] table [x index={7}, y index={3}, col sep=comma, restrict x to domain=0:50] {./Results/ShellExample/smooth_cvar_0.5_norm.csv};
\addlegendentry{$\CVaR_{0.50}$};
\addplot[color3, ultra thick] table [x index={7}, y index={3}, col sep=comma, restrict x to domain=0:50] {./Results/ShellExample/smooth_cvar_0.75_norm_alt.csv};
\addlegendentry{$\CVaR_{0.75}$};
\addplot[color4, ultra thick] table [x index={7}, y index={3}, col sep=comma, restrict x to domain=0:50] {./Results/ShellExample/smooth_cvar_0.8_norm.csv};
\addlegendentry{$\CVaR_{0.8}$};
\addplot[color5, ultra thick] table [x index={7}, y index={3}, col sep=comma, restrict x to domain=0:50] {./Results/ShellExample/smooth_cvar_0.9_norm.csv};
\addlegendentry{$\CVaR_{0.9}$};
\legend{}; %
\end{axis}
\end{tikzpicture}
		\end{scaletikzpicturetowidth}
	\end{minipage}%
	\caption{
		Optimization logs for the design optimization problem of shell structures~\cref{eq:shell-opt-problem} with $\beta=0,~0.5,~0.75,~0.8$, and $0.9$.
		In all cases, the algorithm parameters $\alpha = 0.1$ and $\theta = 0.8$ were used.
		On the left, we see the convergence in the value of the objective function vs. the iteration number.
		In the middle, we see the convergence in the value of the objective function vs. the cumulative number of gradient evaluations.
		On the right, we see the growth in the sample size vs. iteration number.
		\label{fig:shell}}
\end{figure}

\smallskip
\paragraph{\revised{Results}}
In \Cref{fig:shell}, we see that the objective value in each stochastic setting decreases significantly throughout the course of optimization.
As usual, the more risk-averse the problem, the more samples are required.
In fact, in the risk-neutral setting, $\beta = 0$, just $10$ samples is easily enough to fulfill \Cref{con:NormTest2_approx} throughout nearly the entire course of the optimization.
This is the reason why the results we present for this experiment begin with fewer than $10$ samples.
However, using fewer than $10$ samples for the risk-averse experiments did not lead to predictable growth in the initial sample sizes.
This is likely because of a large error in estimating the corresponding quantiles using~\cref{eq:optimal_t} with very few samples.

\begin{figure}
	\centering
	\begin{subfigure}{0.48\textwidth}
		\includegraphics[clip=true, trim= 5cm 0cm 3cm 0cm, width=\textwidth]{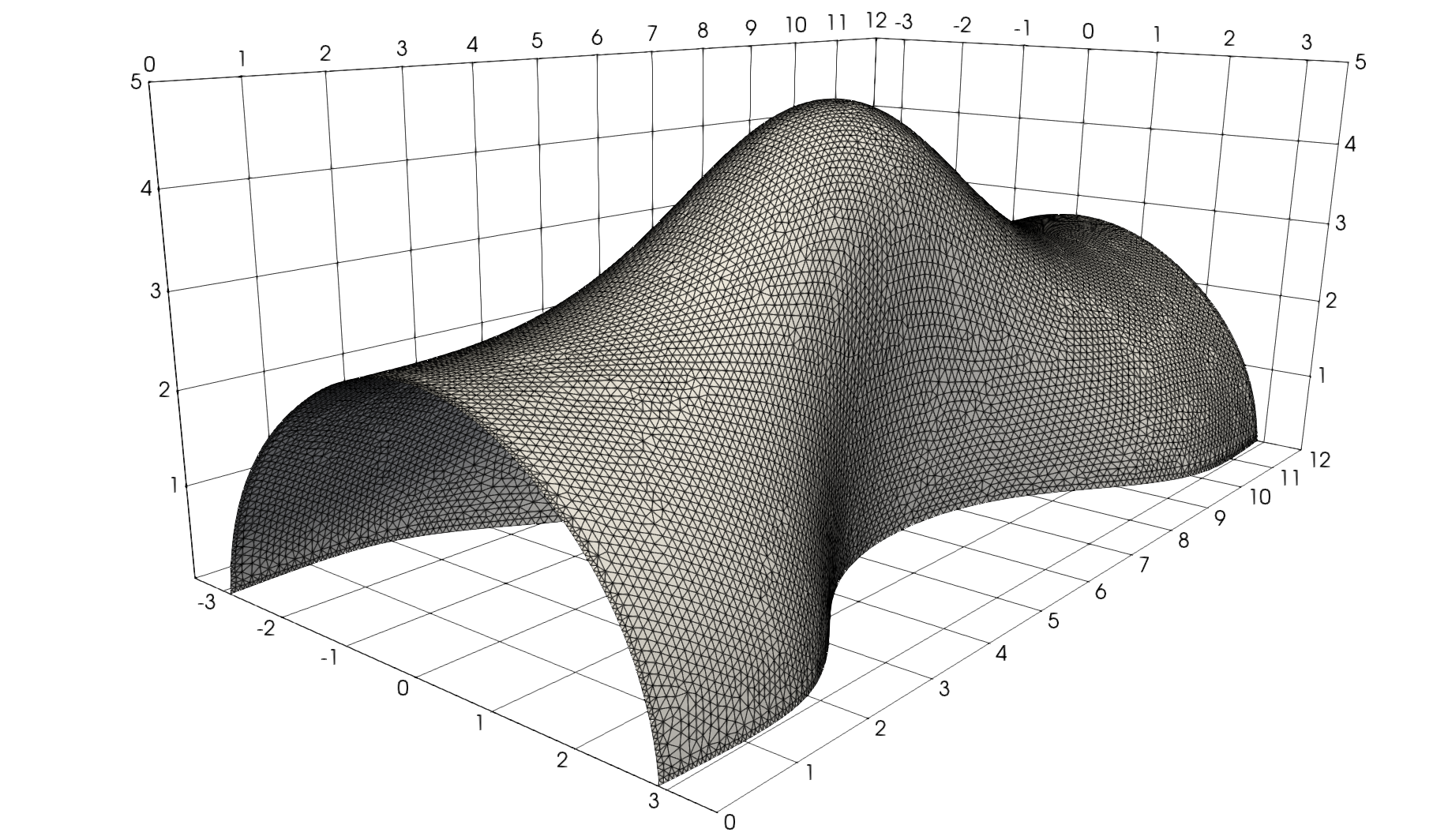}
		\caption{Final shell design when $a = \SI{0.0}{\m}$, $t = \SI{5.00}{\cm}$.}
		\label{fig:shell-deterministic}
	\end{subfigure}
	~
	\begin{subfigure}{0.48\textwidth}
		\includegraphics[clip=true, trim= 5cm 0cm 3cm 0cm, width=\textwidth]{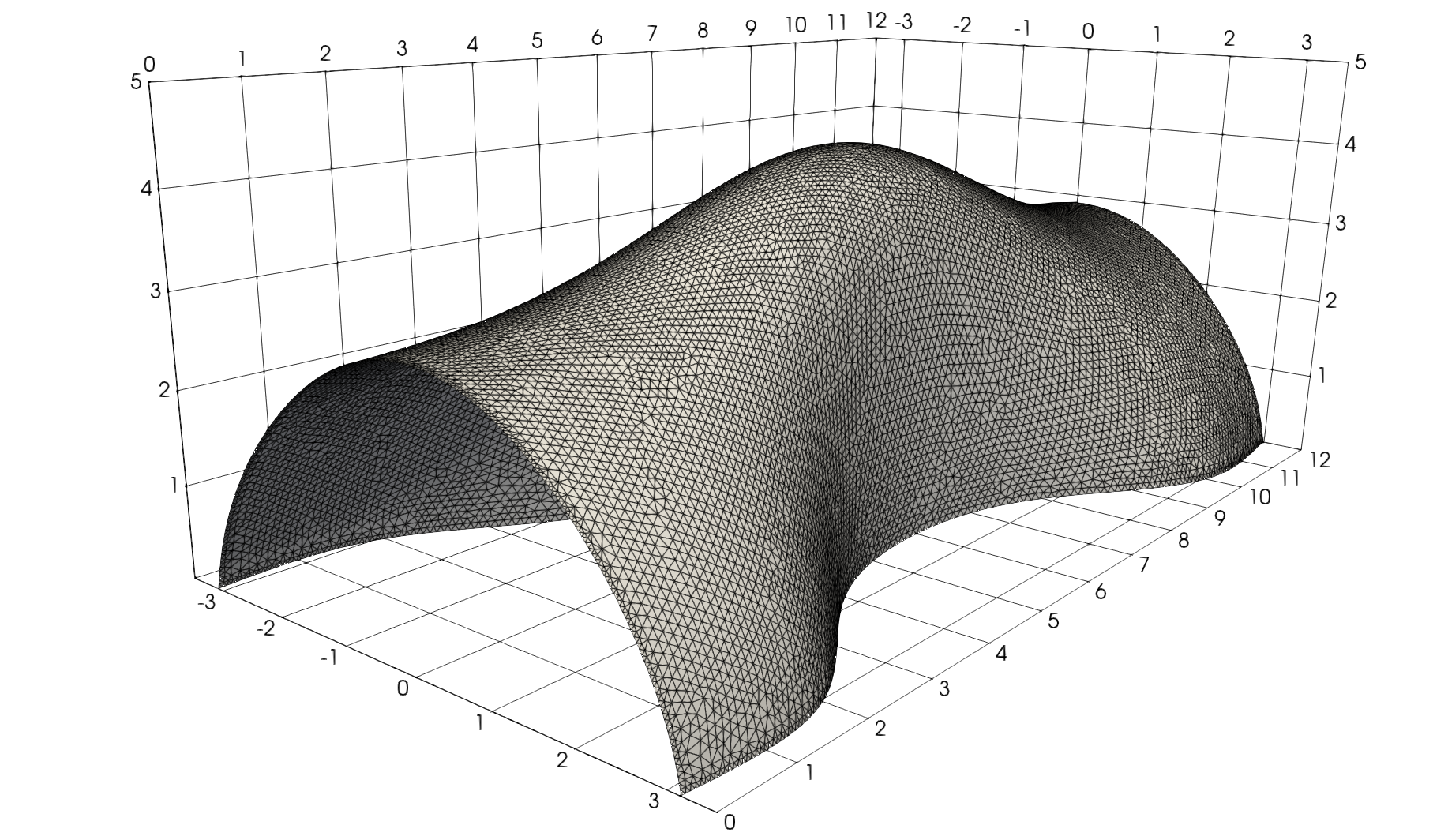}
		\caption{Final shell design when $\mcR = \bbE$}
		\label{fig:shell-E}
	\end{subfigure}
	\\
	\begin{subfigure}{0.48\textwidth}
		\includegraphics[clip=true, trim= 5cm 0cm 3cm 0cm, width=\textwidth]{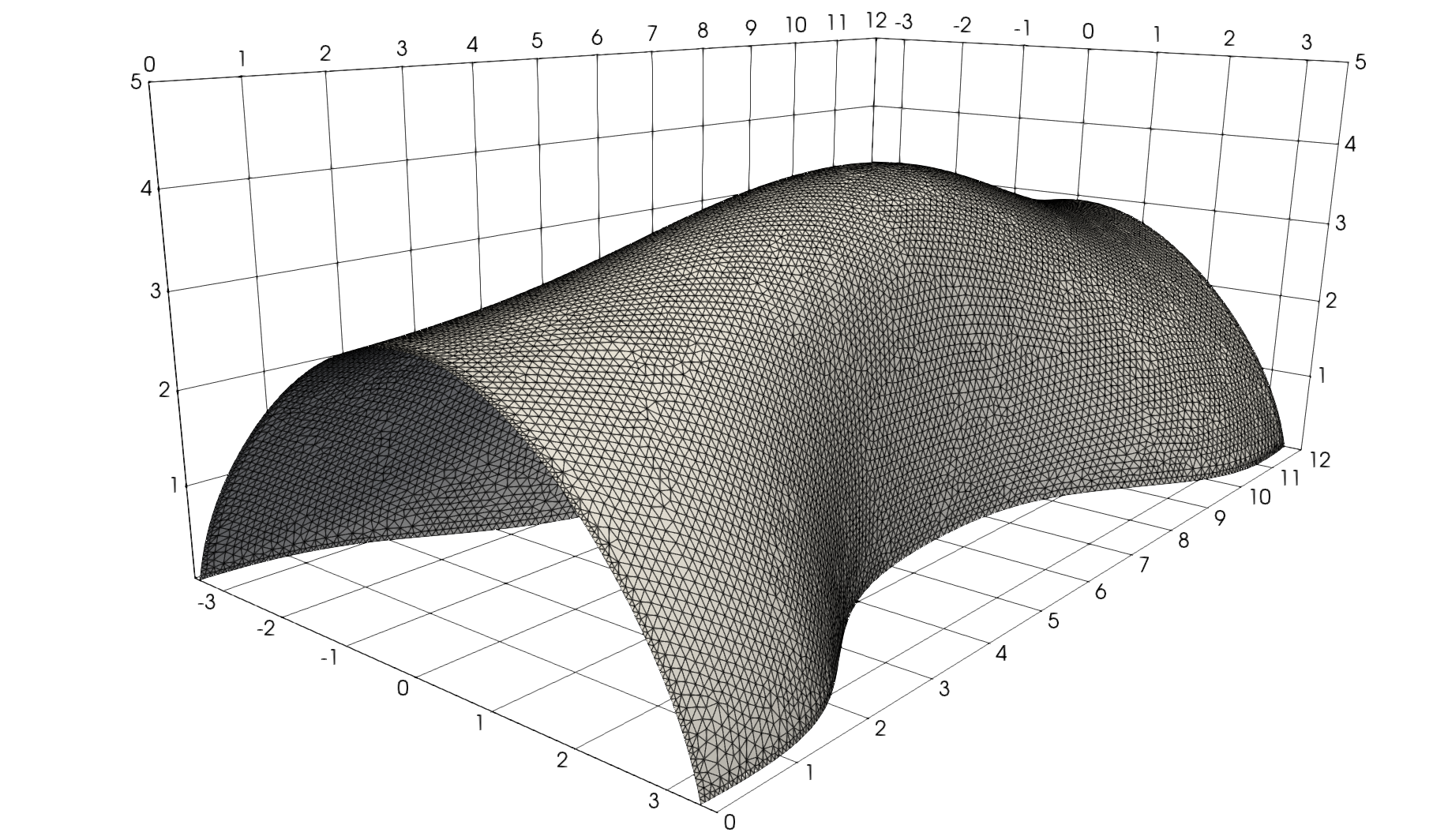}
		\caption{Final shell design when $\mcR = \CVaR_{0.5}$}
		\label{fig:shell-CVaR050}
	\end{subfigure}
	~
	\begin{subfigure}{0.48\textwidth}
		\includegraphics[clip=true, trim= 5cm 0cm 3cm 0cm, width=\textwidth]{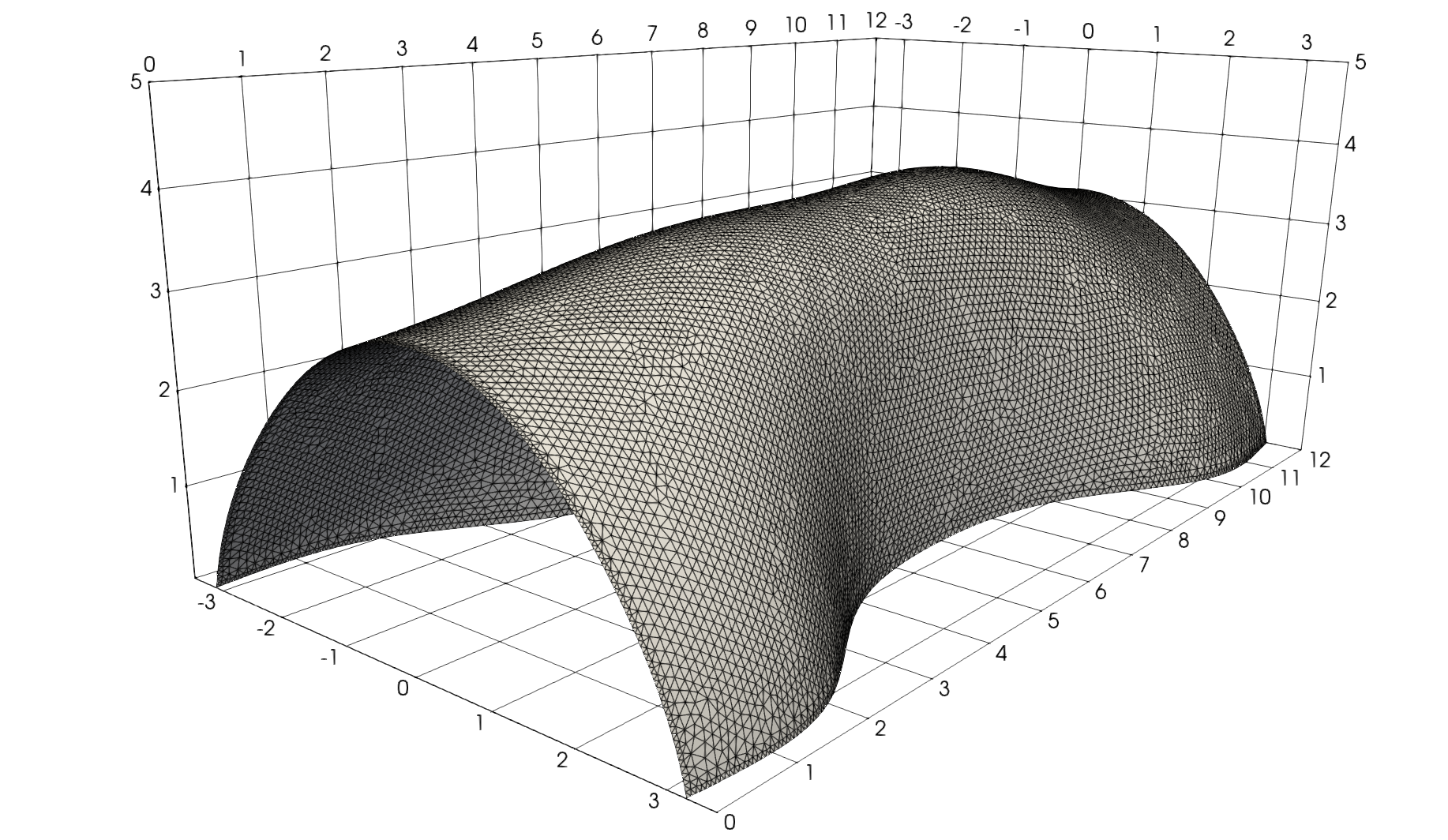}
		\caption{Final shell design when $\mcR = \CVaR_{0.75}$}
		\label{fig:shell-CVaR075}
	\end{subfigure}
	\\
	\begin{subfigure}{0.48\textwidth}
		\includegraphics[clip=true, trim= 5cm 0cm 3cm 0cm, width=\textwidth]{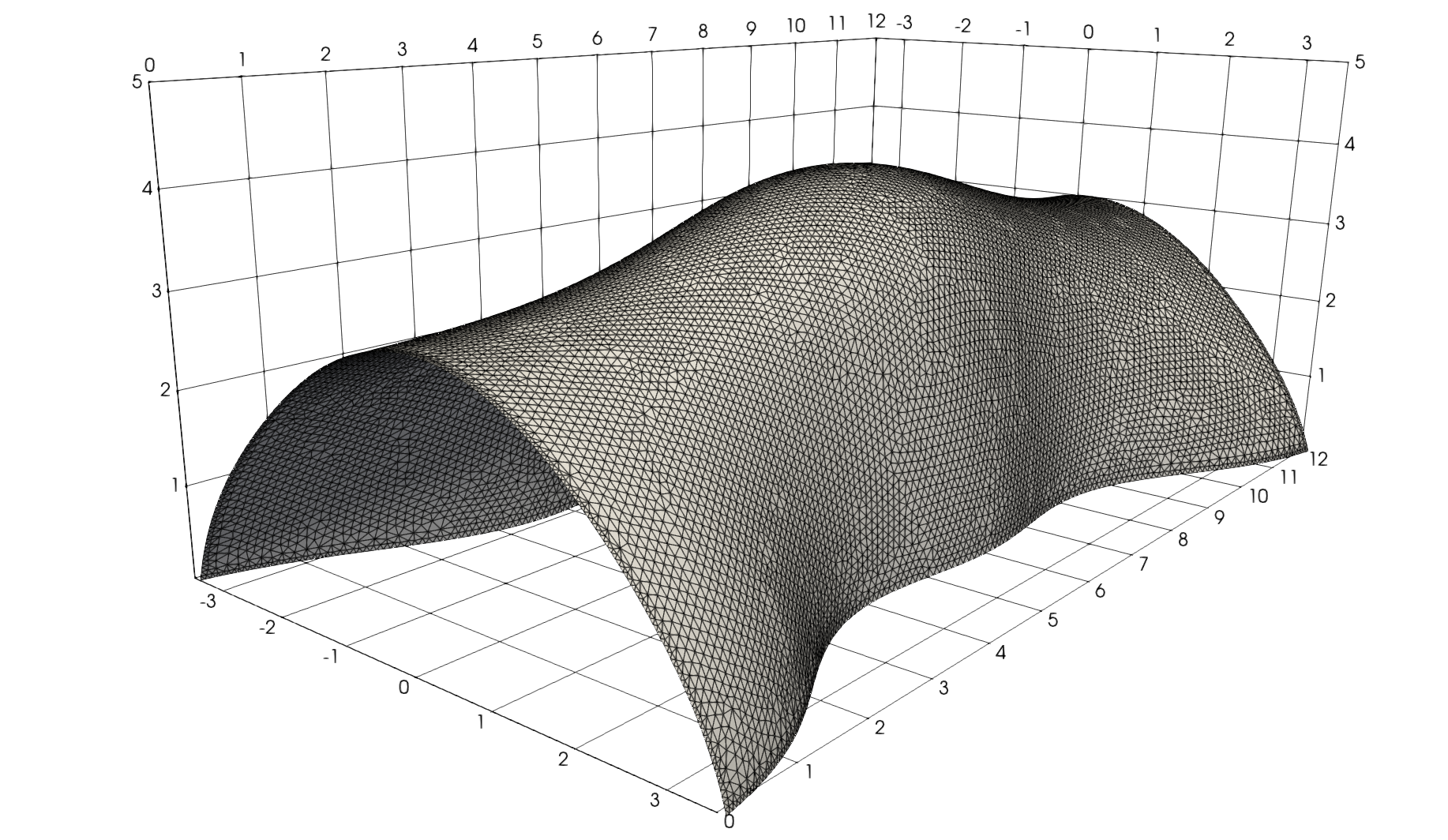}
		\caption{Final shell design when $\mcR = \CVaR_{0.8}$}
		\label{fig:shell-CVaR080}
	\end{subfigure}
	~
	\begin{subfigure}{0.48\textwidth}
		\includegraphics[clip=true, trim= 5cm 0cm 3cm 0cm, width=\textwidth]{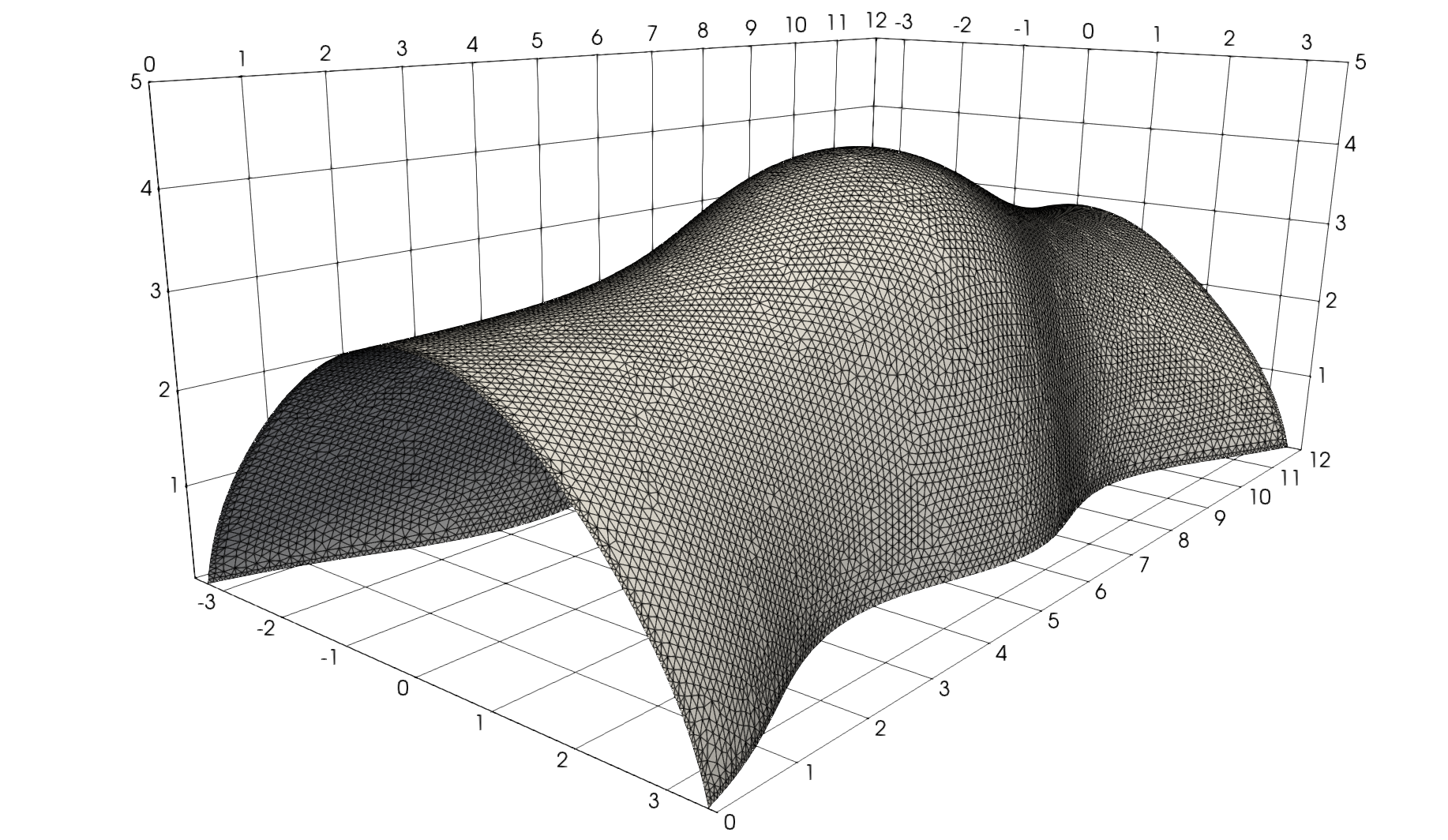}
		\caption{Final shell design when $\mcR = \CVaR_{0.9}$}
		\label{fig:shell-CVaR090}
	\end{subfigure}
	\caption{Final shell designs for deterministic, risk-neutral, and risk-averse optimization problems.\label{fig:shell-final}}
\end{figure}

\revised{It remains to investigate} the performance of \revised{the correction term $g_\its$ introduced in \Cref{alg:adaptive_sampling_SQP} to enforce the equality constraint in} the numerical examples presented \revised{above}. In \Cref{fig:constraint}, the geometry's surface area is shown for the expectation case, that is, $\mathcal{R} = \mathbb{E}$. The values of $A(x_\its)$ are normalized by the initial surface area $A(x_0)$ and scaled by 100\% so that, in an ideal scenario, we would observe a constant value of 100\% as is indicated by the dotted line.
In practice, one notes that we do not satisfy the constraint well after the first step. This is due to the fact that no correction term is computed for this step, as our correction procedure relies on two previous values. However, as the optimization progresses, the normalized surface area stabilizes around 100\% as desired. Similar results were achieved for every other choice of risk measure \revised{considered in the numerical examples of this section}.

\begin{figure}
	\centering
	\begin{minipage}{0.3\textwidth}
		\centering
		\begin{scaletikzpicturetowidth}{\textwidth}
			\begin{tikzpicture}[scale=\tikzscale,font=\large]
\begin{axis}[
xlabel={Iteration},
ylabel={$A(x_k) / A(x_0) \cdot 100\%$},
xmajorgrids,
ymajorgrids,
legend cell align={left},
legend style={fill=white, fill opacity=0.6, draw opacity=1, text opacity=1},
legend style={at={(axis cs:50,15)},anchor=east},
ytick={98.5, 98.75, 99, 99.25, 99.5, 99.75, 100, 100.25, 100.5, 100.75, 101},
]
\addplot[mark=none, black, ultra thick, dotted] coordinates {(0,100) (50,100)};
\addplot[color1, ultra thick] table [x index={0}, y index={2}, col sep=comma, restrict x to domain=0:50] {./Results/ShellConstraint/percentage_expectation.csv};
\addlegendentry{$\E$};
\legend{}; %
\end{axis}
\end{tikzpicture}
		\end{scaletikzpicturetowidth}
	\end{minipage}%
	\caption{Normalized surface area in the \revised{shape optimization} experiment for $\mathcal{R} = \mathbb{E}$. The values were normalized by the initial geometry's surface area and are recorded as percentages.}
	\label{fig:constraint}
\end{figure}

It is interesting to note that the various optimization problems deliver visually distinct optimal shapes.
The risk-neutral design in \Cref{fig:shell-E} is better-suited to the more likely but less damaging loads centered on the middle of the structure.
This is most evident when it is compared to the optimal design in \Cref{fig:shell-deterministic}, which comes from the deterministic scenario where the load is centered on middle of the structure with absolute certainty.
On the other hand, the risk-averse designs \Cref{fig:shell-CVaR050,fig:shell-CVaR075,fig:shell-CVaR080,fig:shell-CVaR090} are progressively better suited to the less likely, but more damaging, loads centered near the ends of the structure.
Initially, as $\beta$ grows, we witness progressively flatter bumps in the center of the structure until a new type of shape appears somewhere around $\beta = 0.75$.
This indicates a change in the local minima landscape.
Likely there are two nearby local minima in this region, one that is a continuation of the risk-neutral minimum and a second local minimum influenced by the extreme values of the stress $\Pi(x)$.
Additional numerical experiments (not included here) indicate to us that the basin of attraction of the new state exhibited in \Cref{fig:shell-CVaR080,fig:shell-CVaR090} grows with $\beta > 0.75$.
At the same time, the basin of attraction of the original risk-neutral category of designs appears to shrink and eventually vanish.

The experiments presented in this subsection highlight one important practical advantage of our adaptive sampling approach over fixed sample size SPGD.
For many engineering applications, gradient evaluations form the majority of the total optimization cost.
This is why we aim to reduce the amount of gradient computations.
In the most risk-averse setting we considered, i.e., $\beta = 0.9$, we observe that the adaptive sample size reaches $|S_k| = 43$ at the final iteration, with $1600$ gradient evaluations in total.
Had we used this sample size from the start, we would have had to evaluate the gradient $2150$ times.
Therefore, in comparison, our adaptive sampling strategy decreased the total number of evaluations by $\sim25\%$.
In other examples, the savings were even greater.
Indeed, when $\beta=0.75$, adaptive sampling required $\sim50\%$ fewer gradient computations than the fixed sample size approach.

\section{Conclusion} %
\label{sec:conclusion}

This paper deals with stochastic optimization algorithms with dynamic sample sizes.
We focus on a large class of stochastic programs with deterministic constraints.
In doing so, we pose sufficient conditions on the sample sizes which guarantee \revised{in the case of convex constrains} that a class of adaptive sampling methods converge.
\revised{We then introduce practical tests to check these conditions and demonstrate their efficiency in a set of numerical examples.}
Our methods not only apply to risk-neutral optimization problems; that is, when the objective function is the expected value of some stochastic quantity.
Indeed, using the conditional value-at-risk (CVaR) as a working example, we show how a large family of risk-averse problems can be treated with the strategies developed here.
\revised{Lastly, we propose an extension for more general non-convex \revisedBW{constraints} and illustrate its robustness in a contemporary application in engineering design with loading and model uncertainties.}

\section*{Acknowledgments}
The authors are grateful for advice and suggestions from the anonymous referees as well as Chandrajit Bajaj and Thomas Surowiec.
The authors acknowledge the computer resources on MareNostrum and Salomon provided by the Barcelona Supercomputing Center (IM-2019-2-0008) and the Czech Republic Ministry of Education, Youth and Sports (from the Large Infrastructures for Research, Experimental Development, and Innovations project ``e-INFRA CZ -- LM2018140''), respectively.
This work was performed under the auspices of the U.S. Department of Energy by Lawrence Livermore National Laboratory under Contract DE-AC52-07NA27344, LLNL-JRNL-827097.

\appendix
\section{More numerical examples} %
\label{appendix:numerical_examples}
To complement the numerical experiments in  \Cref{sec:applications}, we present two short parameter studies. \revised{These illustrate the influence of the choice of step size and sampling rate on the optimization results presented earlier}. 

\subsection{Basic example: The step size}
\label{app:stepsize}

The analysis in \Cref{sec:adaptive_sampling} suggests that we choose a fixed step size $\alpha$ that satisfies~\cref{eq:StepSizeConditionStronglyConvex}. 
\revised{For the example set-up in \Cref{subsec:basic_example}, we can estimate the Lipschitz constant $L \leq 40$ and for given $a_l,~l=1,\dots,20$ calculate also a sharper constant.}
There, we have presented numerical results with $\alpha=0.025$ \revised{corresponding to the rough estimate on $L$.}
In \Cref{fig:sensitivity_stepsize}, we present further experiments to demonstrate how alternative stepsizes influence the results.

\begin{figure}
	\centering
	\begin{minipage}{0.3\textwidth}
		\centering
		\begin{scaletikzpicturetowidth}{\textwidth}
			\begin{tikzpicture}[scale=\tikzscale,font=\large]
\begin{axis}[
xlabel={Iteration},
ylabel={$\lVert x^* - x_k \rVert$},
xmajorgrids,
ymajorgrids,
ymode = log,
legend pos=north east,
legend style={fill=white, fill opacity=0.6, draw opacity=1, text opacity=1, font=\normalsize}
]

\addplot[color3, ultra thick] table [x index={0}, y index={3}, col sep=comma] {./Results/Appendix/statistics_alpha0010.csv};
\addlegendentry{$\alpha=0.010$};

\addplot[color5, ultra thick, restrict x to domain=0:150] table [x index={0}, y index={3}, col sep=comma] {./Results/Appendix/statistics_alpha0025.csv};
\addlegendentry{$\alpha=0.025$};

\addplot[color4, ultra thick, restrict x to domain=0:150] table [x index={0}, y index={3}, col sep=comma] {./Results/Appendix/statistics_alpha0050.csv};
\addlegendentry{$\alpha=0.050$};

\end{axis}
\end{tikzpicture}
		\end{scaletikzpicturetowidth}
	\end{minipage}%
	~
	\begin{minipage}{0.3\textwidth}
		\centering
		\begin{scaletikzpicturetowidth}{\textwidth}
			\begin{tikzpicture}[scale=\tikzscale,font=\large]
\begin{semilogyaxis}[
xlabel={Iteration},
ylabel={Sample size},
xmajorgrids,
ymajorgrids,
ymode = log,
legend style={fill=white, fill opacity=0.6, draw opacity=1, text opacity=1},
legend pos=north west
]

\addplot[color3, ultra thick, restrict x to domain=0:150] table [x index={0}, y index={1}, col sep=comma] {./Results/Appendix/statistics_alpha0010.csv};
\addlegendentry{$\alpha=0.010$};

\addplot[color5, ultra thick, restrict x to domain=0:150] table [x index={0}, y index={1}, col sep=comma] {./Results/Appendix/statistics_alpha0025.csv};
\addlegendentry{$\alpha=0.025$};

\addplot[color4, ultra thick, restrict x to domain=0:150] table [x index={0}, y index={1}, col sep=comma] {./Results/Appendix/statistics_alpha0050.csv};
\addlegendentry{$\alpha=0.050$};

\legend{}; %
\end{semilogyaxis}
\end{tikzpicture}
		\end{scaletikzpicturetowidth}
	\end{minipage}%
	~
	\begin{minipage}{0.3\textwidth}
		\centering
		\begin{scaletikzpicturetowidth}{\textwidth}
			\begin{tikzpicture}[scale=\tikzscale,font=\large]
\begin{axis}[
xlabel={Gradient evaluations},
ylabel={$F(x_k)-F^*$},
xmajorgrids,
ymajorgrids,
xmode = log,
ymode = log,
legend style={fill=white, fill opacity=0.6, draw opacity=1, text opacity=1},
legend pos=north east,
]

\addplot[color3, ultra thick, restrict x to domain=0:150] table [x index={2}, y index={6}, col sep=comma] {./Results/Appendix/statistics_alpha0010.csv};
\addlegendentry{$\alpha=0.010$};

\addplot[color5, ultra thick, restrict x to domain=0:150] table [x index={2}, y index={6}, col sep=comma] {./Results/Appendix/statistics_alpha0025.csv};
\addlegendentry{$\alpha=0.025$};

\addplot[color4, ultra thick, restrict x to domain=0:150] table [x index={2}, y index={6}, col sep=comma] {./Results/Appendix/statistics_alpha0050.csv};
\addlegendentry{$\alpha=0.050$};

\legend{}; %
\end{axis}
\end{tikzpicture}

		\end{scaletikzpicturetowidth}
	\end{minipage}%
	\caption{Numerical results for step sizes $\alpha=0.010,0.025,0.050$ while all other paramethers are kept fixed as in \Cref{subsec:basic_example} with $\theta=0.5$.\label{fig:sensitivity_stepsize}}
\end{figure}

Qualitatively, q-linear convergence is observed for each stepsize in \Cref{fig:sensitivity_stepsize}.
Nevertheless, the largest stepsize ($\alpha = 0.050$) results in both the most iteration-efficient and sample-efficient choice.
Although \Cref{alg:practical} is used here, these results are in line with the analysis in \Cref{thm:StronglyConvex}.

\subsection{Portfolio optimization: The sampling rate}
\label{app:sampling_rate}

In \Cref{ssub:risk_averse_portfolio_optimization}, we have analyzed different choices of the risk-averseness parameter $\beta$. 
Distinct choices of $\beta$ change the optimization problem and, therefore, we used a different adaptive sampling rate $\theta$ for each $\beta$.
We now study the influence of $\theta$ for three different fixed values of $\beta$; namely, $\beta = 0,0.5,0.9$.
Since \Cref{alg:alternative_cvar} performed better than \Cref{alg:practical} in \Cref{subsec:operations_research}, we will focus our investigation on that particular algorithm when $\beta>0$.
The remaining problem set-up is identical to \Cref{ssub:risk_averse_portfolio_optimization}. 

\begin{figure}
	\centering
	\begin{minipage}{0.3\textwidth}
		\centering
		\begin{scaletikzpicturetowidth}{\textwidth}
			\begin{tikzpicture}[scale=\tikzscale,font=\large]
\begin{axis}[
xlabel={Iteration},
ylabel={$F_\beta(x_k,t)$},
xmajorgrids,
ymajorgrids,
legend style={fill=white, fill opacity=0.9, draw opacity=1, text opacity=1},
legend pos=north east,
ymin=-1.225,
ymax=-1.05
]

\addplot[color2, ultra thick] table [x index={0}, y index={3}, col sep=comma,] {./Results/Appendix/statistics_000_theta10.csv};
\addlegendentry{$\theta = 1.0$};

\addplot[color3, ultra thick] table [x index={0}, y index={3}, col sep=comma] {./Results/Appendix/statistics_000_theta30.csv};
\addlegendentry{$\theta = 3.0$};

\addplot[color8, ultra thick] table [x index={0}, y index={3}, col sep=comma] {./Results/Appendix/statistics_000_theta50.csv};
\addlegendentry{$\theta = 5.0$};

\end{axis}
\end{tikzpicture}
		\end{scaletikzpicturetowidth}
	\end{minipage}%
	~
	\begin{minipage}{0.3\textwidth}
		\centering
		\begin{scaletikzpicturetowidth}{\textwidth}
			\begin{tikzpicture}[scale=\tikzscale,font=\large]
\begin{axis}[
xlabel={Gradient evaluations},
ylabel={$F_\beta(x_k,t)$},
xmajorgrids,
ymajorgrids,
xmode = log,
legend style={fill=white, fill opacity=0.9, draw opacity=1, text opacity=1, font=\small},
legend pos=north east,
ymin=-1.225,
ymax=-1.05
]

\addplot[color2, ultra thick] table [x index={2}, y index={3}, col sep=comma,] {./Results/Appendix/statistics_000_theta10.csv};
\addlegendentry{$\theta = 1.0$};

\addplot[color3, ultra thick] table [x index={2}, y index={3}, col sep=comma] {./Results/Appendix/statistics_000_theta30.csv};
\addlegendentry{$\theta = 3.0$};

\addplot[color8, ultra thick] table [x index={2}, y index={3}, col sep=comma] {./Results/Appendix/statistics_000_theta50.csv};
\addlegendentry{$\theta = 5.0$};

\legend{}; %
\end{axis}
\end{tikzpicture}
		\end{scaletikzpicturetowidth}
	\end{minipage}%
	~
	\begin{minipage}{0.3\textwidth}
		\centering
		\begin{scaletikzpicturetowidth}{\textwidth}
			\begin{tikzpicture}[scale=\tikzscale,font=\large]
\begin{axis}[
xlabel={Iteration},
ylabel={Sample size},
xmajorgrids,
ymajorgrids,
ymode = log,
legend style={fill=white, fill opacity=0.6, draw opacity=1, text opacity=1},
legend pos=north west,
ymax=15000000
]

\addplot[color2, ultra thick] table [x index={0}, y index={1}, col sep=comma,] {./Results/Appendix/statistics_000_theta10.csv};
\addlegendentry{$\theta = 1.0$};

\addplot[color3, ultra thick] table [x index={0}, y index={1}, col sep=comma] {./Results/Appendix/statistics_000_theta30.csv};
\addlegendentry{$\theta = 3.0$};

\addplot[color8, ultra thick] table [x index={0}, y index={1}, col sep=comma] {./Results/Appendix/statistics_000_theta50.csv};
\addlegendentry{$\theta = 5.0$};

\legend{}; %
\end{axis}
\end{tikzpicture}
		\end{scaletikzpicturetowidth}
	\end{minipage}%
	\caption{Numerical results for the portfolio optimization problem~\cref{eq:finance-E} using \Cref{alg:practical} for the sampling rates $\theta=1.0,3.0,5.0$. Note that this problem is equivalent to~\cref{eq:finance-CVaR} with $\beta=0.0$. \label{fig:sensitivity000}}
\end{figure}

\begin{figure}
	\centering
	\begin{minipage}{0.3\textwidth}
		\centering
		\begin{scaletikzpicturetowidth}{\textwidth}
			\begin{tikzpicture}[scale=\tikzscale,font=\large]
\begin{axis}[
xlabel={Iteration},
ylabel={$F_\beta(x_k,t)$},
xmajorgrids,
ymajorgrids,
legend style={fill=white, fill opacity=0.9, draw opacity=1, text opacity=1},
legend pos=north east,
ymin=-0.85,
ymax=-0.5
]

\addplot[color2, ultra thick] table [x index={0}, y index={3}, col sep=comma,] {./Results/Appendix/statistics_050_theta10_eps001.csv};
\addlegendentry{$\theta = 1.0$};

\addplot[color3, ultra thick] table [x index={0}, y index={3}, col sep=comma] {./Results/Appendix/statistics_050_theta30_eps001.csv};
\addlegendentry{$\theta = 3.0$};

\addplot[color8, ultra thick] table [x index={0}, y index={3}, col sep=comma] {./Results/Appendix/statistics_050_theta50_eps001.csv};
\addlegendentry{$\theta = 5.0$};

\end{axis}
\end{tikzpicture}
		\end{scaletikzpicturetowidth}
	\end{minipage}%
	~
	\begin{minipage}{0.3\textwidth}
		\centering
		\begin{scaletikzpicturetowidth}{\textwidth}
			\begin{tikzpicture}[scale=\tikzscale,font=\large]
\begin{axis}[
xlabel={Gradient evaluations},
ylabel={$F_\beta(x_k,t)$},
xmajorgrids,
ymajorgrids,
xmode = log,
legend style={fill=white, fill opacity=0.9, draw opacity=1, text opacity=1, font=\small},
legend pos=north east,
ymin=-0.85,
ymax=-0.5
]

\addplot[color2, ultra thick] table [x index={2}, y index={3}, col sep=comma,] {./Results/Appendix/statistics_050_theta10_eps001.csv};
\addlegendentry{$\theta = 1.0$};

\addplot[color3, ultra thick] table [x index={2}, y index={3}, col sep=comma] {./Results/Appendix/statistics_050_theta30_eps001.csv};
\addlegendentry{$\theta = 3.0$};

\addplot[color8, ultra thick] table [x index={2}, y index={3}, col sep=comma] {./Results/Appendix/statistics_050_theta50_eps001.csv};
\addlegendentry{$\theta = 5.0$};

\legend{}; %
\end{axis}
\end{tikzpicture}
		\end{scaletikzpicturetowidth}
	\end{minipage}%
	~
	\begin{minipage}{0.3\textwidth}
		\centering
		\begin{scaletikzpicturetowidth}{\textwidth}
			\begin{tikzpicture}[scale=\tikzscale,font=\large]
\begin{axis}[
xlabel={Iteration},
ylabel={Sample size},
xmajorgrids,
ymajorgrids,
ymode = log,
legend style={fill=white, fill opacity=0.6, draw opacity=1, text opacity=1},
legend pos=north west,
ymax=15000000
]

\addplot[color2, ultra thick] table [x index={0}, y index={1}, col sep=comma,] {./Results/Appendix/statistics_050_theta10_eps001.csv};
\addlegendentry{$\theta = 1.0$};

\addplot[color3, ultra thick] table [x index={0}, y index={1}, col sep=comma] {./Results/Appendix/statistics_050_theta30_eps001.csv};
\addlegendentry{$\theta = 3.0$};

\addplot[color8, ultra thick] table [x index={0}, y index={1}, col sep=comma] {./Results/Appendix/statistics_050_theta50_eps001.csv};
\addlegendentry{$\theta = 5.0$};

\legend{}; %
\end{axis}
\end{tikzpicture}
		\end{scaletikzpicturetowidth}
	\end{minipage}%
	\caption{Numerical results for the portfolio optimization problem~\cref{eq:finance-CVaR} with $\beta=0.5$ using \Cref{alg:alternative_cvar} for the sampling rates $\theta=1.0,3.0,5.0$. \label{fig:sensitivity050}}
\end{figure}

\begin{figure}
	\centering
	\begin{minipage}{0.3\textwidth}
		\centering
		\begin{scaletikzpicturetowidth}{\textwidth}
			\begin{tikzpicture}[scale=\tikzscale,font=\large]
\begin{axis}[
xlabel={Iteration},
ylabel={$F_\beta(x_k,t)$},
xmajorgrids,
ymajorgrids,
legend style={fill=white, fill opacity=0.9, draw opacity=1, text opacity=1},
legend pos=north east,
ymin=-0.4,
ymax=0.0
]

\addplot[color2, ultra thick] table [x index={0}, y index={3}, col sep=comma,] {./Results/Appendix/statistics_090_theta10_eps001.csv};
\addlegendentry{$\theta = 1.0$};

\addplot[color3, ultra thick] table [x index={0}, y index={3}, col sep=comma] {./Results/Appendix/statistics_090_theta30_eps001.csv};
\addlegendentry{$\theta = 3.0$};

\addplot[color8, ultra thick] table [x index={0}, y index={3}, col sep=comma] {./Results/Appendix/statistics_090_theta50_eps001.csv};
\addlegendentry{$\theta = 5.0$};

\end{axis}
\end{tikzpicture}
		\end{scaletikzpicturetowidth}
	\end{minipage}%
	~
	\begin{minipage}{0.3\textwidth}
		\centering
		\begin{scaletikzpicturetowidth}{\textwidth}
			\begin{tikzpicture}[scale=\tikzscale,font=\large]
\begin{axis}[
xlabel={Gradient evaluations},
ylabel={$F_\beta(x_k,t)$},
xmajorgrids,
ymajorgrids,
xmode = log,
legend style={fill=white, fill opacity=0.9, draw opacity=1, text opacity=1, font=\small},
legend pos=north east,
ymin=-0.4,
ymax=0.0
]

\addplot[color2, ultra thick] table [x index={2}, y index={3}, col sep=comma,] {./Results/Appendix/statistics_090_theta10_eps001.csv};
\addlegendentry{$\theta = 1.0$};

\addplot[color3, ultra thick] table [x index={2}, y index={3}, col sep=comma] {./Results/Appendix/statistics_090_theta30_eps001.csv};
\addlegendentry{$\theta = 3.0$};

\addplot[color8, ultra thick] table [x index={2}, y index={3}, col sep=comma] {./Results/Appendix/statistics_090_theta50_eps001.csv};
\addlegendentry{$\theta = 5.0$};

\legend{}; %
\end{axis}
\end{tikzpicture}
		\end{scaletikzpicturetowidth}
	\end{minipage}%
	~
	\begin{minipage}{0.3\textwidth}
		\centering
		\begin{scaletikzpicturetowidth}{\textwidth}
			\begin{tikzpicture}[scale=\tikzscale,font=\large]
\begin{axis}[
xlabel={Iteration},
ylabel={Sample size},
xmajorgrids,
ymajorgrids,
ymode = log,
legend style={fill=white, fill opacity=0.6, draw opacity=1, text opacity=1},
legend pos=north west,
ymax=15000000
]

\addplot[color2, ultra thick] table [x index={0}, y index={1}, col sep=comma,] {./Results/Appendix/statistics_090_theta10_eps001.csv};
\addlegendentry{$\theta = 1.0$};

\addplot[color3, ultra thick] table [x index={0}, y index={1}, col sep=comma] {./Results/Appendix/statistics_090_theta30_eps001.csv};
\addlegendentry{$\theta = 3.0$};

\addplot[color8, ultra thick] table [x index={0}, y index={1}, col sep=comma] {./Results/Appendix/statistics_090_theta50_eps001.csv};
\addlegendentry{$\theta = 5.0$};

\legend{}; %
\end{axis}
\end{tikzpicture}
		\end{scaletikzpicturetowidth}
	\end{minipage}%
	\caption{Numerical results for the portfolio optimization problem~\cref{eq:finance-CVaR} with $\beta=0.9$ using \Cref{alg:alternative_cvar} for the sampling rates $\theta=1.0,3.0,5.0$. \label{fig:sensitivity090}}
\end{figure}

In \Cref{fig:sensitivity000,fig:sensitivity050,fig:sensitivity090}, the influence of the sampling rate $\theta$ becomes visible: the sample size grows faster as $\theta$ decreases. 
The rate of growth of the sample size becomes more pronounced as $\beta$ grows. At the same time, the growth pattern becomes smoother and there are fewer big jumps in the sample size for one iteration to the next.
\phantomsection%
\bibliographystyle{siamplain}
\bibliography{StochasticOptimization}

\end{document}